%% file: GOESingularValues.tex
\documentclass[pdftex]{my-ws-rmta}

\usepackage{verbatim} 

\usepackage{graphicx}
\usepackage[usenames, dvipsnames]{color}
\usepackage{tikz}
\usetikzlibrary{plotmarks}
\usepackage{pgfplots}
\pgfplotsset{compat=newest}
\newlength\figurewidth

\usepackage{listings}
\usepackage{upquote} 

\usepackage[activate={true,nocompatibility},spacing,kerning]{microtype}

\input{glyphtounicode} \pdfgentounicode=1

\usepackage[pdftex,%
            hyperfigures,%
            hidelinks,%
            pdftitle={The Singular Values of the GOE},%
            pdfdisplaydoctitle=true,%
            breaklinks=false,%
            bookmarks=true,%
            unicode=true,
            pdffitwindow=true,%
            pdfauthor={Folkmar Bornemann and Michael LaCroix},%
            pdfsubject={Random Matrices},%
            pdfkeywords={random matrices, GOE, anti-GUE, LUE, singular
              values, determinants, gap probabilities},%
]{hyperref} 

\newsavebox{\gwbox}\newlength{\gwidth}

\newcommand{\defdelim}[1]{\ifx\relax#1\relax\def\lsize{\left}\def\rsize{\right}\else\def\lsize{#1}\def\rsize{#1}\fi}
\newcommand{\floor}[2][]{\defdelim{#1}\lsize\lfloor#2\rsize\rfloor}
\newcommand{\ceil}[2][]{\defdelim{#1}\lsize\lceil#2\rsize\rceil}
\newcommand{\abs}[2][]{\defdelim{#1}\lsize\lvert#2\rsize\rvert}

\def\<#1>{\left\langle{#1}\right\rangle}

\newcommand{\E}[2][]{\operatorname{E}_{#1}\left(#2\right)}
\newcommand{\Var}[2][]{\operatorname{Var}_{#1}\left(#2\right)}

\renewcommand{\leq}{\leqslant}
\renewcommand{\geq}{\geqslant}

\DeclareMathOperator{\diag}{\ensuremath{\mathrm{diag}}}
\DeclareMathOperator{\sign}{\ensuremath{\mathrm{sign}}}
\DeclareMathOperator{\erf}{\ensuremath{\mathrm{erf}}}
\DeclareMathOperator{\even}{\ensuremath{\mathrm{even}}}
\DeclareMathOperator{\odd}{\ensuremath{\mathrm{odd}}}
\newcommand{\GOE}{\ensuremath{\mathrm{GOE}}}
\newcommand{\GSE}{\ensuremath{\mathrm{GSE}}}

\newcommand{\GUE}{\ensuremath{\mathrm{GUE}}}
\newcommand{\aGUE}{\ensuremath{\mathrm{aGUE}}}
\newcommand{\LUE}{\ensuremath{\mathrm{LUE}}}

\newcommand{\R}{\ensuremath{\mathbb{R}}}
\newcommand{\N}{\ensuremath{\mathbb{N}}}

\title{The Singular Values of the \boldmath\GOE{}\unboldmath}

\def\draftnote{\today\quad\currenttime\quad FILE\qquad\jobname}

\begin{document}

\markboth{F. Bornemann and M. La\,Croix}
{The Singular Values of the \GOE{}}

\author{Folkmar Bornemann}
\address{Zentrum Mathematik -- M3, 
  Technische Universit\"at M\"unchen\\
  Boltzmannstr. 3, 85748 Garching bei M\"unchen, Germany\\
\email{bornemann@tum.de}}

\author{Michael La$\,$Croix}
\address{Department of Mathematics,
Massachusetts Institute of Technology\\
77 Massachusetts Avenue, 
Cambridge, MA 02139, USA\\ 
\email{malacroi@mit.edu}}

\maketitle 


\begin{abstract} 
  As a unifying framework for examining several properties that
  nominally involve eigenvalues, we present a particular structure of
  the singular values of the Gaussian orthogonal ensemble (\GOE{}):
  the even-location singular values are distributed as the positive
  eigenvalues of a Gaussian ensemble with chiral unitary symmetry, 
  while the odd-location singular values, conditioned on the even-location ones, can be
  algebraically transformed into a set of independent
  $\chi$-distributed random variables.  We discuss three applications
  of this structure: first, there is a pair of bidiagonal square
  matrices, whose singular values are jointly distributed as the even-
  and odd-location ones of the \GOE{}; second, the magnitude of the
  determinant of the \GOE{} is distributed as a product of simple independent random
  variables; third,  on symmetric intervals, the gap probabilities of the \GOE{}
  can be expressed in terms of the Laguerre unitary ensemble. 
  We work specifically with matrices of finite order, but
  by passing to a large matrix limit, we also obtain new insight into
  asymptotic properties such as the central limit theorem of the
  determinant or the gap probabilities in the bulk-scaling limit.  
  The analysis in this paper avoids much of the technical
  machinery (e.g. Pfaffians, skew-orthogonal
  polynomials, martingales, Meijer $G$-function, etc.) that was previously used to analyze
  some of the applications.
\end{abstract}



\input{Introduction}
\input{JointDensity}
\input{LinearAlgebraBackground}

\input{NewMatrixModel}
\input{DeterminantApplications}

\input{DeterminantCLT}
\input{OrderedConeAnalysis}

\input{GapProbability}

\section*{Acknowledgements}

We would like to thank Peter Forrester for communicating to us his
proof of the odd dimensional $k=0$ case of  \eqref{eqn:gap} and for
encouraging our work.  We would also like to thank Alan Edelman and
David Jackson for asking so many of the right questions.  It would
have been easy to stop with a partial theory of determinants, if they
hadn't anticipated the parallels between the \GOE{} and \GUE{}.  This
research was supported by the DFG-Collaborative Research Center, TRR
109, ``Discretization in Geometry and Dynamics.''  Funding for the
second author was also provided by the National Science Foundation
through grants DMS--1035400 and DMS--1016125.

\bibliographystyle{amsplain}
\bibliography{GOESingularValues}

\end{document}

%% file: Introduction.tex

\section{Introduction}\label{section:introduction}

This paper studies the structure of the singular values of the
Gaussian orthogonal ensemble (\GOE), using it as a unifying framework
for examining several properties that nominally involve eigenvalues.
Here, the $\GOE_n$ of order $n$ is the ensemble of real symmetric
random matrices $$G = (X+X')/2,$$ where $X$ is an $n\times n$ Gaussian
matrix with all entries independent standard normals.  Since the
singular values of symmetric matrices are the magnitudes of the
eigenvalues, the ensemble of singular values will be briefly denoted
by $|\GOE_n|$. Central to our discussion is the immediate set decomposition,
 \begin{equation}\label{eqn:decomp}
 |\GOE_n| = \even |\GOE_n| \;\cup\; \odd |\GOE_n|,
\end{equation}
of the ordered singular values according to the parity of their indices,
where the even-location decimated ensemble $\even |\GOE_n|$ is defined
by taking the 2nd largest, 4th largest, etc.  singular value, and
similarly for $\odd |\GOE_n|$. 

Our \emph{first} set of main results relates the decomposition \eqref{eqn:decomp} to the eigenvalues 
of a Gaussian ensemble with chiral, or
anti-symmetric, unitary symmetry. Namely, with $X$ as above, the
ensemble of real skew-symmetric random matrices $$A=(X-X')/2$$ will be
called the anti-GUE with its (almost surely) different and positive
singular values written briefly as $\aGUE_n$ (if $n$ is odd, there
is a surplus singular value zero, which is omitted).\footnote{In
  this paper, the Gaussian weights are $e^{-\beta x^2/2}$ with
  $\beta=1$ for orthogonal and $\beta=2$ for unitary symmetry.} 
Then, the following structure holds.

\begin{theorem}\label{thm:main1} Denoting equality of the joint
  distribution by $\overset{\rm d}{=}$, there holds
\begin{equation}\label{eqn:even_aGUE}
\even |\GOE_n| \overset{\rm d}{=} \aGUE_n.
\end{equation}
\end{theorem}

\noindent
We will give two proofs that differ in their handling of the
odd-location singular values: one (Section~\ref{section:newmodel}) by
algebraically transforming them to a set of {\em independent} positive
variables, each distributed as~$\chi_2$ and, if $n$ is odd, a surplus
$\chi_1$; the other (Section~\ref{section:evensingularvalues}) by
integrating them out. Both proofs are based on an algebraic
factorization (Section~\ref{section:jointdensity}) of the joint density
of $|\GOE_n|$, where one factor depends only on the even-location singular
values, the other on the odd-location ones. If we recall the
superposition representation, see \cite[Eq.~(2.6)]{EvenSymm} or
\cite[Thm.~1]{EL-GUEsing}, 
\[
|\GUE_n| \overset{\rm d}{=}  \aGUE_n\,\cup\; \aGUE_{n+1},
\]
of the singular values of the Gaussian
unitary ensemble (GUE),
with both ensembles on the right drawn independently,
Theorem~\ref{thm:main1} immediately implies the following remarkable
relation between the singular values of GUE and GOE:
\begin{corollary} With the ensembles on the right drawn independently, there holds
\begin{equation}\label{eqn:super}
|\GUE_n|  \overset{\rm d}{=}  \even |\GOE_n| \,\cup\, \even |\GOE_{n+1}|.
\end{equation}
\end{corollary}

\noindent
The superposition \eqref{eqn:super} bears a striking
similarity with a corresponding result for the eigenvalue
distributions, see \cite[Eq.~(5.9)]{FR} and \cite[Eq.~(6.14)]{FR2005}, namely
\[
\GUE_n \overset{\rm d}{=}  \even\left( \GOE_n \cup\, \GOE_{n+1}\right).
\]

Our \emph{second} set of main results
(Sections~\ref{section:newmodel}/\ref{section:determinant}) sharpens
Theorem~\ref{thm:main1} by realizing $|\GOE_n|$ as the singular values
of other matrix models that reveal a rich additional structure.  A first
model (Corollary~\ref{cor:newmodel-chi_n}), initially identified by
comparing moments of the product of the even singular values to known
moments of the determinant \cite[Eqs.~(23)
and~(24)]{Andrews-Goulden-Jackson}, is constructed by bordering the
skew-symmetric matrix $A$ defining the anti-GUE with an independent
standard normal vector $b\in\R^n$: that is to say, the singular values
of $G$ and those of $$H = (b \; A)$$ are both distributed as
$|\GOE_n|$. To the same end, the bordering vector could also be chosen
as $b=\tau_n e_1$, where $\tau_n$ is a $\chi_n$-distributed variable,
independent of $A$, and $e_1$ denotes the first unit vector. The
precise effect of such borderings on the singular values of a matrix
is studied in the preparatory Section~\ref{section:linearalgebra}.

Using a technique (Lemma~\ref{lem:square}) that was, in essence,
introduced by Dumitriu and Forrester \cite{Dumitriu}, this bordered
matrix model is finally
(Theorems~\ref{thm:square-bidiagonal-even}/\ref{thm:square-bidiagonal-odd})
transformed into a pair $(R^{\even}, R^{\odd})$ of \emph{bidiagonal}
square matrices, whose singular values are jointly distributed as
$\even|\GOE_n|$ and $\odd|\GOE_n|$.  Both matrices depend in a very
simple fashion on a set of independent $\chi_k$-distributed random
variables.  Specifically, for $n=2m$ even (the structure of the odd
order case is similar), we get
\begin{gather*}
R^{\odd}  = 
\frac{1}{\sqrt{2}}\begin{pmatrix}
    \sqrt{\xi_1^2 + 2\xi_{2m}^2} & \xi_{2m-2}  \\
    & \xi_{2m-1} & \xi_{2m-4} \\
    & &\ddots & \ddots  \\
    & & &\xi_5  & \xi_2 \\
    & & & & \xi_3
  \end{pmatrix}\\
\intertext{and}
R^{\even} = 
\frac{1}{\sqrt{2}}
  \begin{pmatrix}
    \xi_1 & \xi_{2m-2}  \\
    & \xi_{2m-1} & \xi_{2m-4} \\
    & &\ddots & \ddots  \\
    & & &\xi_5  & \xi_2 \\
    & & & & \xi_3
  \end{pmatrix},
  \end{gather*}
where $\xi_1,\xi_2,\ldots,\xi_{2m}$ are independent random variables, with $\xi_k$ distributed as $\chi_k$.
The singular values of $R^{\odd}$ correspond to $\odd |\GOE_{2m}|$ and those of $R^{\even}$ to $\even|\GOE_{2m}|$,
both drawn from the same ensemble.

As a striking application (Corollary~\ref{cor:detfact}) of this new
matrix pair model, we establish that $|\!\det \GOE_n|$ can be
expressed explicitly as a product of independent random
variables. Specifically, for $n=2m$ even, the determinant of
$M=\sqrt{2}G$ factors as
\begin{equation}\label{eq:detMeven}
|\!\det M| = \xi_1 \sqrt{\xi_1^2+2\xi_{2m}^2} \cdot \xi_3^2 \cdot \xi_5^2 \,\cdots\, \xi_{2m-1}^2,
\end{equation}
with independent variables $\xi_k$ distributed as $\chi_k$.
A similar factorization holds in the odd order case.
The form of these variables explains the absence of large prime
factors in the moments of the determinant, and leads
to a new, simple proof (Section~\ref{section:CLT}) of the known central limit theorem for
$\log|\!\det\GOE_n|$, cf. Delannay and Le~Ca\"er
\cite[Section~III]{Delannay-LeCaer}, Tao and Vu \cite[Thm.~4]{TaoVu}. While
 the representation \eqref{eq:detMeven} of $|\!\det M|$ as a product of independent random variables can
 be found implicitly in the work of Delannay and Le~Ca\"er, namely in form of a factorization \cite[Eq.~(41)]{Delannay-LeCaer}
of the Meijer $G$-function
representation of the Mellin transform of $|\!\det M|$ into hypergeometric terms, see the discussion of \eqref{eqn:fancy-mellin},
Tao and Vu, who approximated the log-determinant by a sum of weakly dependent terms, speculated that
such a representation would not be possible \cite[p.~78]{TaoVu}. 

Our \emph{third} set of main results (Section~\ref{section:gapprobabilities}) studies the implication
of Theorem~\ref{thm:main1} on the inter-relation of gap probabilities, that is, the
probabilities $E(k;J)$ that the interval $J$ contains exactly $k$
eigenvalues drawn from a random matrix ensemble. Specifically, for
order $n$, we get
\[
E^{n}_{\GOE}(2k+\mu-1;(-s,s)) + E_{\GOE}^{n}(2k+\mu;(-s,s)) = E_{\aGUE}^{n}(k;(0,s)),
\]
where $\mu=0,1$ denotes the parity of $n$. This formula was previously
known only in the case $\mu=0$, see Forrester
\cite[Eq.~(1.14)]{EvenSymm}.  
We initially used a heuristic argument, see \eqref{eqn:heuristic}, to
  extrapolate the formula to the case $\mu=1$.  A substantial portion of
  the present discussion was derived from attempts to justify the
  heuristic, after this prediction held up under numerical scrutiny. 
Taking the bulk scaling
limit of both cases provides a new, simpler proof of a remarkable
formula previously obtained by Mehta relating the gap probabilities of the
\GOE{} and those of the Laguerre unitary ensemble (\LUE{}), see \eqref{eqn:bulk}.

\paragraph{\bf Notation.} In contrast to the previous analyses
mentioned, where either ensembles of odd (e.g., if Pfaffians were
used) or of even order (e.g., if Mellin transforms were used) have
typically presented considerable technical complications, our
treatment of ensembles of even and odd order is nearly identical. The
formulae themselves, however, will often depend on the parity~$\mu$ of
the underlying order $n$ and we will, throughout this paper, write
\begin{subequations}\label{eqn:dim}
\begin{equation}
  n = 2m + \mu \quad (\mu=0,1),\qquad \hat{m}=m+\mu,
\end{equation}
that is,  
\begin{equation}
  m = \floor{n/2}, \qquad \hat{m} = \ceil{n/2}, 
  \qquad \mu = \ceil{n/2} - \floor{n/2}.
\end{equation}
\end{subequations}
Terms that only appear for $n$ odd will be written with a factor $\mu$
in a sum and with an exponent $\mu$ in a product; etc.  This way,
without suggesting any natural interpolation between the cases $\mu=0$
and $\mu=1$ (with the notable exception of the usage of the heuristic
duality principle \eqref{eqn:heuristic} that started our work), we
simply avoid writing out awkward case distinctions.

%% file: JointDensity.tex

\section{Joint Density of the Singular Values}\label{section:jointdensity}

In this section we establish, in two different ways, the joint
probability distribution of the singular values $\sigma_j=\abs[
]{\lambda_j}$ of the \GOE{} induced by the corresponding density for
eigenvalues
as given by the symmetric function
\begin{equation}\label{eqn:goedensity}
  p(\lambda_1,\lambda_2,\dotsc,\lambda_n)=c_n\prod_{k=1}^n
  \mathrm{e}^{-\lambda_k^2/2}\cdot
  \abs{\Delta(\lambda_1,\lambda_2,\dotsc,\lambda_n)}
\end{equation}
with some normalization constant $c_n$ and the Vandermonde determinant
\[
\Delta(\xi_1,\ldots, \xi_n) = \det\begin{pmatrix}
1 & 1 & \cdots & 1 \\
\xi_1 & \xi_2 & \cdots & \xi_n\\
\vdots & \vdots & & \vdots\\
\xi_1^{n-1} & \xi_2^{n-1}& \cdots & \xi_n^{n-1} 
\end{pmatrix} = \prod_{k>j} (\xi_k-\xi_j).
\]
We will frequently use that $\Delta(\xi_1,\ldots, \xi_n)\geq0$ if the arguments are
increasingly ordered, $\xi_1 \leq \cdots \leq \xi_n$.

By symmetry, we can establish the joint density of the singular values
by restricting ourselves to the cone of increasingly ordered singular
values
\begin{equation}\label{eqn:cone}
0 \leq \sigma_1 \leq \dotsb \leq \sigma_n,
\end{equation}
this way parametrizing  $|\GOE_n|$.
To simplify notation and to avoid case distinctions between odd and
even order $n$ in later parts of the paper, we
introduce two further sets of coordinates for this cone.
Writing, as detailed in \eqref{eqn:dim},
$n = 2m + \mu$ and $\hat{m}=m+\mu$ with $\mu=0,1$,
the coordinates
\begin{subequations}\label{eqn:coord_xy}
\begin{equation}\label{eqn:coord_xya}
x_j = \sigma_{2j-1}\quad (j=1,\ldots,\hat{m}), \qquad y_j = \sigma_{2j} \quad (j=1,\ldots,m)
\end{equation}
satisfy the interlacing property
\begin{equation}\label{eqn:interlacing_xy}
0 \leq x_1 \leq y_1 \leq x_2 \leq y_2 \leq \dotsb \leq x_{\hat{m}} \leq y_{\hat{m}},
\end{equation}
\end{subequations}
formally adding the value $y_{m+1}=\infty$ if $\mu=1$. With
$x^\downarrow$ and $y^\downarrow$ denoting the $x$ and $y$ vectors
with their components taken in the reverse order, so
$x^\downarrow=(x_{\hat{m}},x_{\hat{m}-1},\dotsc,x_1)$ and
$y^\downarrow=(y_m,y_{m-1},\dotsc,y_1)$, we define, depending on the parity
of $n$, the coordinates
\begin{subequations}\label{eqn:coord_st}
\begin{equation} \label{eqn:coord_st_a}
(t,s)=(y^\downarrow,x^\downarrow) \quad (\mu=0),\qquad (t,s)=(x^\downarrow,y^\downarrow) \quad (\mu=1),
\end{equation}
satisfying the interlacing property
\begin{equation}\label{eqn:interlacing}
t_1 \geq s_1 \geq t_2 \geq s_2 \geq \cdots \geq t_{\hat m} \geq s_{\hat m} \geq 0,
\end{equation}
\end{subequations}
again formally adding the value $s_{m+1}=0$ if $\mu=1$.   
A large part of the apparent dependence on parity
  is the fact that some results, like Theorem~\ref{thm:JointDensity},
  have stable expressions in terms of the $(x,y)$ coordinates, while
  others, like Theorem~\ref{thm:independent}, are stable in the
  $(t,s)$ coordinates.
Since the
mapping from $\sigma=(\sigma_1,\ldots,\sigma_n)$ to either the pair of
coordinates $(x,y)$ or $(t,s)$ is orthogonal, transforming the density
between the three sets of coordinates is simply done by inserting new
variable names for old ones. Note that the $s$ variables parametrize the even-location decimated ensemble $\even |\GOE_n|$
while the $t$-variables do the same for $\odd|\GOE_n|$. We call them the even and odd singular values.

Supported on the cone defined by \eqref{eqn:cone}, the joint probability density
of the singular values is 
\[
q(\sigma_1,\ldots,\sigma_n) = n! \sum_{\epsilon \in \{\pm 1\}^n}
p(\epsilon_1 \sigma_1,\ldots,\epsilon_n \sigma_n) = c_n n!
\prod_{k=1}^n e^{-\sigma_k^2/2} D(\sigma_1,\ldots,\sigma_n)
\]
with 
\[
 D(\sigma_1,\ldots,\sigma_n) = \sum_{\epsilon \in \{\pm 1\}^n}
 |\Delta(\epsilon_1 \sigma_1,\ldots, \epsilon_n \sigma_n)|.
\]
To determine the signs of the Vandermonde terms it suffices to discuss the case $\sigma_1 < \cdots < \sigma_n$:
we then get, because $\sign ( \epsilon_k \sigma_k - \epsilon_j \sigma_j )=
\epsilon_k$ if $k>j$, 
\[
\sign \Delta(\epsilon_1 \sigma_1,\ldots, \epsilon_n \sigma_n) =
\prod_{k>j} \epsilon_k = \prod_{k=2}^n \epsilon_k^{k-1} =
\prod_{\text{$k$ even}} \epsilon_k.
\]
Hence, by continuity, there holds on all of (\ref{eqn:cone})
\begin{equation}\label{eqn:D_without_abs}
D(\sigma_1,\ldots,\sigma_n) = \sum_{\epsilon \in \{\pm 1\}^n}
\theta_0(\epsilon)\,\Delta(\epsilon_1 \sigma_1,\ldots, \epsilon_n
\sigma_n),\qquad \theta_0(\epsilon) = \prod_{\text{$k$ even}}
\epsilon_k.
\end{equation}
The form of $\theta_0(\epsilon)$ suggests we proceed in
terms of the $(x,y)$ coordinates introduced in (\ref{eqn:coord_xy}).
 With respect to these coordinates,
we obtain the following theorem.

\begin{theorem}\label{thm:JointDensity} 
The joint probability density of $|\GOE_n|$,
supported on the cone~\eqref{eqn:cone} and expressed in the coordinates
\eqref{eqn:coord_xy}, is given by
\begin{equation}\label{eqn:JointDensity}
c_n n! 2^n \cdot \left( \prod_{k=1}^{\hat{m}} e^{-x_k^2/2} \cdot
\Delta(x_1^2,\ldots,x_{\hat{m}}^2) \right)\cdot \left( \prod_{k=1}^m
y_k e^{-y_k^2/2} \cdot \Delta(y_1^2,\ldots,y_{m}^2)\right),
\end{equation}
where $c_n$ is the normalization constant of the \GOE-density
\eqref{eqn:goedensity}.
\end{theorem}

\begin{remark}
  Despite the fact that the joint density factors on its domain of
  support, it does \emph{not} reveal an independence between the
  underlying variables $x$ and $y$. Their dependence is entirely by
  the interlacing (\ref{eqn:interlacing_xy}).
\end{remark}

We will give two different proofs of the theorem.  The first uses the
determinantal structure of the Vandermonde terms to establish the
factorization, while the second uses their polynomial structure and their symmetries.  It
is our consideration that the first proof is more straightforward,
while the second provides additional insight into the structure of the
factorization.

\input{DetProof}

\input{PolynomialProof}


%% file: DetProof.tex

\subsection*{Proof by Determinantal Structure}\label{sec:densityfirstproof}

We write $D(x;y)$ for (\ref{eqn:D_without_abs}) when expressed in terms of the $(x,y)$ variables (\ref{eqn:coord_xy}); 
it is convenient to split the sign changes $\epsilon$ into $\epsilon^x$ and $\epsilon^y$
accordingly and to use 
\[
\theta_0(\epsilon) = \theta(\epsilon^y), \qquad 
\theta(\epsilon^y) = \epsilon^y_1\cdots\epsilon_m^y.
\] 

Since determinants are invariant with respect to
a simultaneous row and column permutation so that the odd columns and rows occur before the even ones, we
express the Vandermonde terms as
\[
\Delta(\epsilon\sigma_1,\ldots,\epsilon\sigma_n) = 
\det
\begin{pmatrix}
\pi^{(\hat{m})}_0(x_1) & \ldots & \pi^{(\hat{m})}_0(x_{\hat{m}}) & \pi^{(\hat{m})}_0(y_1) & \ldots & \pi^{(\hat{m})}_0(y_m) \\*[2mm]
\epsilon^x_1\pi^{(m)}_1(x_1) & \ldots & \epsilon^x_{\hat{m}}\pi^{(m)}_1(x_{\hat{m}}) & 
\epsilon^y_1\pi^{(m)}_1(y_1) & \ldots & \epsilon^y_{m}\pi^{(m)}_1(y_m)
\end{pmatrix}
\]
by writing the determinant column-wise with
\[
\pi_\mu^{(n)}(x) = \begin{pmatrix}
x^{\mu} \\
x^{\mu+2}\\
\vdots\\
x^{\mu+2n-2}
\end{pmatrix} \in \R^n\qquad (\mu=0,1).
\]
Now, we calculate
\begin{multline*}
 D(x;y)
 =\\*[2mm]  
\sum_{\substack{\epsilon^x \in\{\pm1\}^{\hat{m}}\\*[1mm]\epsilon^y \in\{\pm1\}^m }} \theta(\epsilon^y) 
 \det
\begin{pmatrix}
\pi^{(\hat{m})}_0(x_1) & \ldots & \pi^{(\hat{m})}_0(x_{\hat{m}}) & \pi^{(\hat{m})}_0(y_1) & \ldots & \pi^{(\hat{m})}_0(y_m) \\*[2mm]
\epsilon^x_1\pi^{(m)}_1(x_1) & \ldots & \epsilon^x_{\hat{m}}\pi^{(m)}_1(x_{\hat{m}}) & 
\epsilon^y_1\pi^{(m)}_1(y_1) & \ldots & \epsilon^y_{m}\pi^{(m)}_1(y_m)
\end{pmatrix}\\*[2mm]
 =   \sum_{\epsilon^y \in\{\pm1\}^m} \theta(\epsilon^y)
\det
\begin{pmatrix}
2\pi^{(\hat{m})}_0(x_1) & \ldots & 2\pi^{(\hat{m})}_0(x_{\hat{m}}) & \pi^{(\hat{m})}_0(y_1) & \ldots & \pi^{(\hat{m})}_0(y_m) \\*[2mm]
0 & \ldots & 0 & 
\epsilon^y_1\pi^{(m)}_1(y_1) & \ldots & \epsilon^y_{m}\pi^{(m)}_1(y_m)
\end{pmatrix}\\*[4mm]
=2^{\hat{m}} 
\det
\begin{pmatrix}
\pi^{(\hat{m})}_0(x_1) & \ldots & \pi^{(\hat{m})}_0(x_{\hat{m}})
\end{pmatrix}
\cdot
\sum_{\epsilon^y \in\{\pm1\}^m} \theta(\epsilon^y)^2
\det
\begin{pmatrix} 
\pi^{(m)}_1(y_1) & \ldots & \pi^{(m)}_1(y_m)
\end{pmatrix}\\*[2mm]
=2^{\hat{m}} 
\det
\begin{pmatrix}
\pi^{(\hat{m})}_0(x_1) & \ldots & \pi^{(\hat{m})}_0(x_{\hat{m}})
\end{pmatrix}
\cdot 2^{m} 
\det
\begin{pmatrix} 
\pi^{(m)}_1(y_1) & \ldots & \pi^{(m)}_1(y_m)
\end{pmatrix}.
\end{multline*}
By noting $\hat{m} + m = n$ and by expressing the result in terms of Vandermonde determinants,  we finally get
\begin{equation}\label{eqn:Dxy}
D(x;y) = 2^n \cdot \Delta(x_1^2,\ldots, x_{\hat{m}}^2) \cdot y_1\cdots y_m\; \Delta(y_1^2,\ldots,y_{m}^2).
\end{equation}
This factorization establishes Theorem~\ref{thm:JointDensity}.

%% file: PolynomialProof.tex

\subsection*{Proof by Polynomiality}\label{sec:densitysecondproof}

We give a second proof of the factorization (\ref{eqn:Dxy}) based on
the observation that the sum in (\ref{eqn:D_without_abs})
defines a polynomial of homogeneous degree at most $\binom{n}{2}$.  We identify
the factors by symmetrizing known vanishings of this polynomial.
Extending the definition of $D$ polynomially to all real values of its
arguments, we get in particular
\begin{equation}\label{eqn:D_signed}
D(\epsilon_1\sigma_1,\ldots,\epsilon_n\sigma_n) = \theta_0(\epsilon)
D(\sigma_1,\ldots,\sigma_n)\qquad (\epsilon \in \{\pm 1\}^n)
\end{equation}
and, inherited from the Vandermonde terms, $D$ is \emph{antisymmetric}
with respect to permutations of either the even or odd indices of
$\sigma$ since both sets of permutations leave the factor
$\theta_0(\epsilon)$ invariant.

Now, if $\sigma_j = \sigma_{j+2}$ for $\sigma$ belonging to the cone
(\ref{eqn:cone}), we have $\sigma_j=\sigma_{j+1}=\sigma_{j+2}$ and,
hence, by the pigeonhole principle, for each choice of signs
$\epsilon$ at least one of
\[
\epsilon_j \sigma_j = \epsilon_{j+1} \sigma_{j+1} \quad\text{or}\quad\epsilon_j \sigma_j = \epsilon_{j+2} \sigma_{j+2} \quad\text{or}\quad\epsilon_{j+1} \sigma_{j+1} = \epsilon_{j+2} \sigma_{j+2}  
\]
holds. Therefore, each of the Vandermonde terms in
(\ref{eqn:D_without_abs}) vanishes. It follows that $\sigma_j -
\sigma_{j+2}$ and by~(\ref{eqn:D_signed}) also $\sigma_j +
\sigma_{j+2}$ divide $D$ for all $j$, thus so does the product
$\sigma_j^2-\sigma_{j+2}^2$. We also note that if $\sigma_2 = 0$, we
have $\sigma_1=\sigma_2=0$ and, hence, for each choice of signs
$\epsilon_1 \sigma_1 = \epsilon_2 \sigma_2$.  Once more each of the
Vandermonde terms in (\ref{eqn:D_without_abs}) vanishes and it follows
that $\sigma_2$ divides $D$.

In terms of the $(x,y)$-coordinates 
(\ref{eqn:coord_xya}) we thus see that $x_j^2-x_{j+1}^2$,
$y_j^2-y_{j+1}^2$ and $y_1$ divide $D(x;y)$. Invoking the antisymmetry
with respect to either $x$ or $y$, we see that $D$ is divisible by
$x_j^2-x_k^2$, $y_j^2-y_k^2$ for every $j\neq k$ and by $y_j$ for
every $j$.  These factors contribute homogeneous degree
$$2\binom{\hat{m}}{2}+2\binom{m}{2}+m=\binom{n}{2},$$ so $D$ cannot have
any other non-unit factors and we get
\begin{equation}\label{eqn:Dxy_dn}
D(x;y) = d_n \cdot \Delta(x_1^2,\ldots, x_{\hat{m}}^2) \cdot y_1\cdots y_m\; \Delta(y_1^2,\ldots,y_{m}^2)
\end{equation}
with a positive constant $d_n$. This is (\ref{eqn:Dxy}) except for identifying $d_n=2^n$.

\begin{remark} The value of $d_n$ can easily be calculated without resorting to the first proof: 
a straightforward inspection shows that the expressions (\ref{eqn:D_without_abs}) and (\ref{eqn:Dxy_dn}) both induce an asymptotics of the form
\[
D(\sigma_1,\ldots,\sigma_{n-1},\sigma_n) \sim \kappa_n \sigma_n^{n-1} D(\sigma_1,\ldots,\sigma_{n-1})\qquad (\sigma_n \to \infty);
\]
the first one gives $\kappa_n = 2$, the second one $\kappa_n = d_n/d_{n-1}$. From $D(\sigma_1) = 2$ we thus get $d_n = 2^n$. 
\end{remark}

%% file: LinearAlgebraBackground.tex

\section{Singular Values of Bordered (Skew-)Symmetric Matrices}\label{section:linearalgebra}

In preparation for what follows, in this section we study an algebraic
device that allows us to untangle the interlacing
\eqref{eqn:interlacing} of even and odd singular values, namely
bordering a \emph{skew-symmetric} matrix $A \in \R^{n \times n}$ with a column vector: $A \mapsto (b\;\; A)$.  By looking at the
purely imaginary Hermitian matrix $i A$ we see that each non-zero
singular value of $A$ occurs with even multiplicity.
That is, with $n=2m+\mu$, 
the $n$ singular values of $A$ can be arranged as the
sequence $s_1,s_1,s_2,s_2,\ldots,s_m,s_m$ and, if $\mu=1$, also
$s_{m+1}=0$, decreasingly ordered according to
\[
s_1 \geq s_2 \geq \cdots \geq s_{\hat m},\qquad \hat m = m + \mu.
\]
The results of this section are twofold. First, Lemma \ref{lem:2:1}
shows that the singular values of $(b\;\; A)$ are of the form (with
the value $s_{\hat m}=0$ only formally added to the list of
inequalities if $\mu=1$)
\begin{equation}
t_1 \geq s_1 \geq t_2 \geq s_2 \geq \cdots \geq t_{\hat m} \geq s_{\hat m}.\tag{\ref{eqn:interlacing}}
\end{equation}
That is, bordering $A$ by a column $b$ splits the double listed pairs $(s_j,s_j)$ of singular values into $(s_j,t_j)$ and, if $\mu=1$,
modifies the surplus singular value $s_{\hat m}=0$ into some $t_{\hat m}$ subject to the interlacing~(\ref{eqn:interlacing}).
They are {\em strictly interlacing}  if
\[
t_1 > s_1 > t_2 > s_2 > \cdots > t_{\hat m} > s_{\hat m}.
\]
Second, Lemma \ref{lem:2:2} establishes a coordinate change $(t,s) \mapsto (r,s)$
through an explicit algebraic map
such that strict interlacing of $t$ with $s$ corresponds to strict positivity of the components of $r$.

To begin with, there are orthogonal matrices
$U$ and $V$ such that (the last row and column of the block partitioning are understood to be $\mu$-dimensional, meaning
that they are missing if $\mu=0$)
\begin{equation}\label{eqn:SVD_A}
U A V' = 
\begin{pmatrix}
S &0 &0 \\
0 &S &0 \\
0 &0 &0
\end{pmatrix},\qquad
U (b\;\; A) \begin{pmatrix}
1 & 0 \\
0 & V'
\end{pmatrix}
= 
\begin{pmatrix}
u & S &0 &0 \\
v & 0 &S &0 \\
\eta &0 &0 & 0
\end{pmatrix},
\end{equation}
with $S = \diag(s_1,\ldots,s_m)$ built from the singular values of $A$ and the partitioning 
\[
Ub =
\begin{pmatrix}
u\\
v\\
\eta
\end{pmatrix},\qquad u,v \in \R^m,\quad \eta \in \R^\mu.
\]
Hence, the singular values of $(b\;\; A)$ are given by the following lemma.

\begin{lemma}\label{lem:2:1} Let $S = \diag(s_1,\ldots,s_m)$ be a diagonal matrix and $u, v \in \R^m$, $\eta \in \R^\mu$ $(\mu=0,1)$. 
Then, with $\hat m = m + \mu$, the singular values of the $(m+\hat m)\times(m+\hat m +1)$ block matrix
\[
\begin{pmatrix}
u &S &0 &0 \\
v &0 &S &0 \\
\eta &0 &0 &0
\end{pmatrix}
\]
are $s_1,\ldots,s_m,t_1,\ldots,t_{\hat m}$, satisfying the interlacing property \eqref{eqn:interlacing}. Here, the $t_j$ are the singular~values of the $\hat m\times (\hat m +1)$ bordered matrix $(r\;\, \hat S)$ with 
$\hat S = \diag (s_1,\ldots,s_{\hat m})$ and $r_j = \sqrt{u_j^2 + v_j^2}$ ($j=1,\ldots,m$);
if $\mu=1$, then $r_{m+1} = |\eta|$ and $s_{m+1}= 0$.
Further, there holds
\begin{equation}\label{eqn:t_sum_squared}
t_1^2 +\cdots + t_{\hat{m}}^2 = r_1^2 +\cdots + r_{\hat{m}}^2 + s_1^2 + \cdots + s_m^2
\end{equation}
and, if $\mu=1$,
\begin{equation}\label{eqn:t_prod}
t_1 \cdots \,t_{\hat{m}} = r_{\hat{m}} \cdot s_1 \cdots \,s_m.
\end{equation}
\end{lemma}

\begin{proof} It suffices to prove that the two matrices
\[
M_1 = \begin{pmatrix}
u &S &0 &0 \\
v &0 &S &0 \\
\eta &0 &0 &0
\end{pmatrix}
, \qquad
M_2 = \begin{pmatrix}
0 &S &0 &0 \\
r &0 &S &0 \\
|\eta| &0 &0 &0
\end{pmatrix}
\]
with $r_j = \sqrt{u_j^2 + v_j^2}$, $j=1,\ldots,m$, have the same singular values. Using a Givens rotation $U_j$ with
\[
U_j 
\begin{pmatrix}
u_j \\
v_j
\end{pmatrix} =
\begin{pmatrix}
0 \\
r_j
\end{pmatrix} \qquad (j=1,\ldots,m)
\]
one gets
\[
U_j 
\begin{pmatrix}
u_j & s_j & 0\\
v_j & 0 & s_j
\end{pmatrix} 
\begin{pmatrix}
1 & 0 \\
0 & U_j'
\end{pmatrix} = \begin{pmatrix}
0 & s_j & 0\\
r_j & 0 & s_j
\end{pmatrix}.
\]
Hence, by successively applying these two-dimensional orthogonal operations to the corresponding rows and columns (and addressing
a possible sign change of the last row if $\mu=1$) one transforms $M_1$ into $M_2$ while leaving the singular values invariant.

We note that the decreasingly ordered singular values $s_1,\ldots,s_{\hat m}$ of a matrix $\hat S$ and those
of the bordered matrix $(r\;\; \hat S)$, $t_1,\ldots,t_{\hat m}$,  
are generally known \cite[Cor.~7.3.6]{HornJohnson} to be {\em interlacing} as in \eqref{eqn:interlacing}.

To finish, \eqref{eqn:t_sum_squared} follows from expressing the Frobenius norm
of $(r\;\,\hat{S})$ in terms of its singular values~$t$ and \eqref{eqn:t_prod} follows from doing the same, if $\mu=1$, for
the magnitude of the determinant of that matrix with the last column (which is all zeros then) deleted.
\end{proof}

The next lemma shows that one can uniquely solve the inverse problem $t \mapsto r$ for strict interlacing.

\begin{lemma}\label{lem:2:2} With the notation as in Lemma~\ref{lem:2:1}, let $s_1 > s_2 > \cdots > s_{\hat m} \geq 0$ with $s_{\hat m} = 0$ if $\mu=1$, and let $\hat S = \diag(s_1,\ldots,s_{\hat m})$. Then,  the map
\[
\Phi: r \mapsto \text{$t =$ the decreasingly ordered singular values of $(r\;\; \hat S)$}
\]
defines a diffeomorphism 
\[
\Phi : \R^{\hat m}_{>0} \to \left\{ t \in \R^{\hat m}_{>0} : \text{$t$ is strictly interlacing with $s$}\right\}.
\]
If $t$ is strictly interlacing with~$s$, its preimage $r = \Phi^{-1}(t)$ is the unique positive
solution of the system
\begin{equation}\label{eqn:secular}
\sum_{k=1}^{\hat m} \frac{r_k^2}{t_j^2 -s_k^2} = 1\qquad (j=1,\ldots,\hat m),
\end{equation}
which is explicitly solved by
\begin{equation}\label{eqn:secular_solution}
r_j^2 = - \frac{\omega_t(s_j^2)}{\omega_s'(s_j^2)}, \qquad \omega_s(\xi) = (\xi-s_1^2)\cdots (\xi-s_{\hat m}^2), \quad \omega_t(\xi) = (\xi-t_1^2)\cdots (\xi-t_{\hat m}^2).
\end{equation}
The Jacobian of the inverse map $\Phi^{-1}$ is given by
\begin{equation}\label{eqn:jacobian}
\det \left(\frac{\partial r_j}{\partial t_k}\right)_{1\leq j,k\leq \hat m}
=  \frac{1}{r_1\cdots r_m} \cdot  
\frac{(t_1  \cdots t_{\hat m})^{1-\mu}\;\Delta(t_{\hat m}^2,\ldots,t_1^2)}
{(s_1  \cdots s_m)^{\mu}\;\Delta(s_m^2,\ldots,s_1^2)} \qquad (\mu=0,1).
\end{equation} 
\end{lemma} 

\begin{proof} The squares of the singular values $t_1,\ldots,t_{\hat m}$ of $(r\;\; \hat S)$ are the eigenvalues of 
\[
(r\;\; \hat S) (r\;\; \hat S)' = \hat S \hat S' + r r' = \diag(s_1^2,\ldots,s_{\hat m}^2) + r r'.
\]
Now, any set of values $t_j^2$ for which $t_j$ satisfies the interlacing property (\ref{eqn:interlacing}) can be obtained in this way,
that is, as the eigenvalues of a positive semi-definite rank-one perturbation of $D=\diag(s_1^2,\ldots,s_{\hat m}^2)$ (see, e.g., \cite[Sect.~2]{Thompson}). Since $r r'$ does not depend on the signs of the individual entries of $r$, we can always choose $r \in \R_{\geq 0}^{\hat m}$. If $r_\nu = 0$ for some $\nu$, then the $\nu$-th row  and the $\nu$-th column of $r r'$ are
zero which means that $s_\nu^2$ appears among the values of $t_j^2$. Hence, \emph{strict} interlacing implies $r\in \R_{>0}^{\hat m}$.

Given such an $r  \in \R_{>0}^{\hat m}$, the eigenvalues $t_j^2$ of $D + r r'$ are known (\cite[Lemma~8.4.3]{Golub}) to be {\em strictly} interlacing with the $s_k^2$  and satisfy the secular equation
\[
f(t_j^2) = 0\quad (j=1,\ldots,\hat m),\qquad f(\lambda) = 1+  r' (D-\lambda I)^{-1}  r,
\]
which is (\ref{eqn:secular}). 
Since the determinant (\ref{eqn:cauchy_det}) given below is non-zero and, hence, the Cauchy matrix 
\[
C = \left(\frac{1}{t_j^2-s_k^2}\right)_{1\leq j,k \leq \hat m}
\]
is non-singular, there is a one-to-one correspondence of $r \in \R^{\hat m}_{>0}$ with those $t$  that \emph{strictly} 
interlace with $s$. Because each of the steps $t \mapsto C \mapsto r$ is smooth, we have therefore proved that $\Phi$ is a diffeomorphism.

By relating Cauchy matrices with Lagrangian polynomial interpolation, Schechter \cite[Eq.~(16)]{Schechter} gave
a short and simple proof of the explicit formula (\ref{eqn:secular_solution}). Differentiation with respect to $t_k$ gives
\[
J_{jk}=\frac{\partial r_j}{\partial t_k}  = \frac{r_j t_k}{s_j^2-t_k^2}.
\]
Hence, $$J = \diag(r_1,\ldots,r_{\hat m}) C \diag(t_1,\ldots,t_{\hat m}),$$ which implies $\det J =  t_1\cdots t_{\hat m} \, r_1\cdots r_{\hat m}\, \det C$.
Now, using the explicit determinantal formula \cite[Eq.~(4)]{Schechter} 
\begin{equation}\label{eqn:cauchy_det}
\det C = \frac{\prod_{j < k} (t_j^2 - t_k^2)(s_k^2-s_j^2)}{\prod_{j,k} (t_j^2-s_k^2)}
\end{equation}
and
\[
r_1^2\cdots r_{\hat m}^2 = (-1)^{\hat m} \frac{\omega_t(s_1^2)\cdots \omega_t(s_{\hat{m}}^2)}{\omega_s'(s_1^2)\cdots \omega_s'(s_{\hat{m}}^2)},
\]
together with the following straightforward evaluations of the product terms
\[
(-1)^{\hat m}\omega_t(s_1^2)\cdots \omega_t(s_{\hat{m}}^2) = \prod_{j,k} (t_j^2 - s_k^2),\qquad 
\omega_s'(s_1^2)\cdots \omega_s'(s_{\hat{m}}^2) = \prod_{j\neq k}(s_j^2-s_k^2),
\]
one gets
\[
\det J = \frac{t_1 \cdots t_{\hat{m}}}{r_1 \cdots r_{\hat{m}}}\cdot  \frac{\prod_{j<k} (t_j^2 - t_k^2)}{\prod_{j < k} (s_j^2 - s_k^2)}.
\]
 With
\begin{gather*}
\prod_{j<k} (t_j^2 - t_k^2) = \Delta(t_{\hat{m}}^2,\ldots,t_1^2),\\*[2mm] \prod_{j < k} (s_j^2 - s_k^2) = \Delta(s_{\hat{m}}^2,\ldots,s_1^2) = (s_1\cdots s_m)^{2\mu} \Delta(s_m^2,\ldots,s_1^2),
\end{gather*}
one finally gets the expression \eqref{eqn:jacobian} by using \eqref{eqn:t_prod} if $\mu=1$.
\end{proof}

%% file: NewMatrixModel.tex

\section{Random Matrix Models for the Odd and Even Singular Values}\label{section:newmodel}

Because of interlacing, the factorization of the joint density stated
in Theorem~\ref{thm:JointDensity} does not reveal an independence
between the $x$ and the $y$ components of the singular values, or to
the same end, between the $t$ and the $s$ components. If we change,
however, the $(t,s)$ coordinates to the $(r,s)$ coordinates
introduced in Lemma~\ref{lem:2:2}, the interlacing is replaced by just
a positivity condition on the $r$ components. The following theorem, which sharpens Theorem~\ref{thm:main1},
shows that not only are the $r$ and the $s$ components independent
of each other but both sets of components have so much additional
structure that they can be completely described in terms of known
distributions.

\begin{theorem}\label{thm:independent} Applying the
transform $(t,s) \mapsto (r,s)$ of Lemma~\ref{lem:2:2} to the $(t,s)$ parametrization \eqref{eqn:coord_st} of the decimated ensembles
$\odd|\GOE_n|$ and $\even|\GOE_n|$
defines a set of random variables  $r_k$,
which are distributed as $\chi_2$ for $k=1,\ldots,m$ and, if $\mu=1$, distributed as
$\chi_1$ for $k=m+1$. They are independent of each other and of the even singular values $s$, which are
jointly distributed~as
\[
\even |\GOE_n| \overset{\rm d}{=} \aGUE_n.
\]
\end{theorem}
\begin{proof}
In terms of the $(t,s)$ coordinates, the joint density \eqref{eqn:JointDensity} of the singular values of the \GOE{} can be recast in the form
\begin{multline}\label{eqn:singular_density}
q(s;t) = c_n 2^n n! \cdot \left(s_1\cdots s_m\right)^\mu \left(t_1\cdots t_{\hat{m}}\right)^{1-\mu} \\*[0mm] \cdot \Delta(s_m^2,\ldots,s_1^2) \Delta(t_{\hat{m}}^2,\ldots,t_1^2) \cdot e^{-\sum_{j=1}^m \frac{s_j^2}{2} - \sum_{j=1}^{\hat{m}}\frac{t_j^2}{2}},
\end{multline}
where the case distinction between even ($\mu=0$) and odd ($\mu=1$) orders $n$ has been expressed in terms of powers. If we
apply the coordinate change $(t,s) \mapsto (r,s)$ of Lemma~\ref{lem:2:2}, which is a diffeomorphism up to an exceptional set of zero probability, the density with respect to the
$(r,s)$ is 
\begin{multline*}
 \left( \det \left(\frac{\partial r_j}{\partial t_k}\right)_{1\leq j,k\leq \hat m}\right)^{-1} \cdot q(s;t)
= r_1\cdots r_m \cdot \frac{(s_1  \cdots s_m)^{\mu}\,\Delta(s_m^2,\ldots,s_1^2)}{(t_1  \cdots t_{\hat m})^{1-\mu}\,\Delta(t_{\hat m}^2,\ldots,t_1^2)}\; q(s;t)\\*[2mm]
= \left(\prod_{j=1}^m r_j e^{-r_j^2/2}\right) \cdot 
\left(\sqrt{\frac{2}{\pi}}\,e^{-r_{\hat{m}}^2/2}\right)^\mu \cdot 
\left( \delta_\mu c_n 2^n n! \cdot \prod_{j=1}^m s_j^{2\mu} e^{-s_j^2} \cdot \Delta(s_m^2,\ldots,s_1^2)^2\right)
\end{multline*}
with $\delta_\mu=\left(\pi/2\right)^{\mu/2}$. Here we used expression (\ref{eqn:jacobian}) for the Jacobian and simplified the exponential functions according to (\ref{eqn:t_sum_squared}). On their supporting domains, the first $m$ factors of the resulting density are a $\chi_2$-density each, the next one is 
a $\chi_1$-density if $\mu=1$ (disappearing if $\mu=0$), and the last one is the joint density of the anti-GUE of order $n$,
see \cite[Sect.~13.1]{Mehta} or \cite[Ex. 1.3.5(iv)]{ForresterBook}.
\end{proof}

\begin{remark} As a side product, the proof shows that the normalization constant $a_n$ of the joint density of the anti-GUE, if extended
by symmetry to be supported on $[0,\infty)^m$, is given by
\[
a_n = c_n \left(\frac{\pi}{2}\right)^{\mu/2} \frac{2^n n!}{m!}\qquad 
(n = 2m+\mu, \mu = 0,1).
\]
This is consistent with the explicit formulae for $c_n^{-1}$ and $a_n^{-1}$ given in \cite[Eq.~(1.163)/Eq.~(4.157)]{ForresterBook}.
\end{remark}

The proof of Theorem~\ref{thm:independent} shows that the joint density $p(t|s)$ of the $t$ variables conditioned on $s$
is given by the expression
\begin{equation}\label{eqn:conditional}
p(t|s) = \frac{1}{\delta_\mu} \frac{\left(t_1\cdots t_{\hat{m}}\right)^{1-\mu}\Delta(t_{\hat{m}}^2,\ldots,t_1^2)}
{\left(s_1\cdots s_m\right)^{\mu}\Delta(s_m^2,\ldots,s_1^2)}   e^{-\sum_{j=1}^{\hat{m}} t_j^2/2 + \sum_{j=1}^m s_j^2/2}.
\end{equation}
This is just a particular case of a general result by Forrester
and Rains \cite[Cor.~3]{FR2005}, which gives the probability $p(t|s)$ if the  $t_k$ are the solutions of the secular equation~(\ref{eqn:secular}) with parameters $r_j$ being independently gamma distributed. In retrospect, we could thus have proved Theorem~\ref{thm:independent}, based on Theorem~\ref{thm:JointDensity}, starting
with the Forrester--Rains formula (\ref{eqn:conditional})  and working backwards.

Now, Theorem~\ref{thm:independent} and Lemma~\ref{lem:2:1} yield a new random matrix model for $|\GOE_n|$ which amounts to the singular values of certain bordered skew-symmetric Gaussian matrices.

\begin{corollary}\label{cor:newmodel-chi_n} Let $X \in \R^{n \times n}$ be a random matrix with independent standard 
normal entries. Denote by $G=(X+X')/2$ its symmetric 
and by $A=(X-X')/2$ its skew-symmetric part. Let $\tau_n$
be a $\chi_n$-distributed random variable that is independent of $X$. Then, both the singular values of $G$ and the singular values of the bordered matrix (with $e_1$ denoting the first unit vector) 
\begin{equation}\label{eqn:H-model}
H=(\tau_n e_1\;\; A)
\end{equation}
are jointly distributed as those of the $\GOE$ of order $n$. The same holds if the matrix $H$ is obtained from bordering 
$A$ with 
an independent standard normal vector.
\end{corollary}

\begin{proof} Note that the singular values of the symmetric part $G=(X+X')/2$ are by {\em definition} jointly
distributed as the singular
values of the \GOE{} of order $n$. 

As discussed in the derivation of (\ref{eqn:SVD_A}), the singular value decomposition of $A$ takes the form
\[
U A V' = 
\begin{pmatrix}
S &0 &0 \\
0 &S &0 \\
0 &0 &0
\end{pmatrix}, \qquad S=\diag(s_1,\ldots,s_m),
\]
where the last row and columns are missing if $n$ is odd. By symmetry, the orthogonal matrix $U$ is Haar distributed---independently
of  $\{s_1,\ldots,s_m\}$, which are jointly distributed as the anti-\GUE{} of order $n$, 
cf. \cite[Sect.~13.1]{Mehta} or \cite[Ex. 1.3.5(iv)]{ForresterBook}. Applying $U$ to the first column $H_1$ of $H$ defines
\[
U H_1 = \tau_n U_1  =
\begin{pmatrix}
u\\
v\\
\eta
\end{pmatrix},\qquad u,v \in \R^m,\quad \eta \in \R^\mu.
\]
Since the first column $U_1$ of $U$ is uniformly distributed on the sphere $S^{n-1}$ and $\tau_n$ is independently $\chi_n$-distributed, we see
that $UH_1$  is a standard normal vector, see, e.g.,
\cite[Sect.~V.4]{Devroye}; the same conclusion holds for standard normal $H_1$. Hence, the
variables 
\[
r_j = \sqrt{u_j^2 + v_j^2} \qquad (j=1,\ldots,m)
\]
are independently $\chi_2$-distributed  and, if $\mu=1$, $r_{m+1} = |\eta|$ is independently  $\chi_1$-distributed.
Comparing the results of Lemma~\ref{lem:2:1} and of Theorem~\ref{thm:independent} finishes the proof.
\end{proof}

The random matrix model of the last corollary can easily be turned
into the following sparse model that separates the even and odd
singular values. Note that just one of the two matrices is square, the
other is rectangular.

\begin{corollary}\label{cor:bidiagonal} Let $n=2m+\mu$, $\mu=0,1$ and let $\tau_1,\ldots,\tau_n$ be independent random variables, with $\tau_k$ distributed as $\chi_k$. The union of the singular values of both the bidiagonal matrix
\[
B_\mu^{\odd} = (\tau_n e_1\;\; B_\mu^{\even}) \in \R^{(m+\mu)\times (m+1)}
\]
and the bidiagonal matrix $B_\mu^{\even}\in \R^{(m+\mu)\times m}$, defined by
\[
B_0^{\even} = \frac{1}{\sqrt{2}}
  \begin{pmatrix}
    \tau_{2m-1}  \\
    \tau_{2m-2} & \tau_{2m-3} \\
    &\ddots & \ddots  \\
    &&\tau_2  & \tau_1
  \end{pmatrix},\qquad
B_1^{\even} = \frac{1}{\sqrt{2}}
  \begin{pmatrix}
    \tau_{2m}  \\
    \tau_{2m-1} & \tau_{2m-2} \\
    &\ddots & \ddots  \\
    &&\tau_3  & \tau_2 \\
    &&& \tau_1
  \end{pmatrix},
  \]
is jointly distributed as $|\GOE_n|$. Here, the singular values of $B_\mu^{\odd}$ correspond to
$\odd|\GOE_n|$ and the singular values of $B_\mu^{\even}$ correspond to $\even|\GOE_n|$, both drawn from the same
ensemble.
\end{corollary}

\begin{proof} Using the notation of Corollary
  \ref{cor:newmodel-chi_n}, a Householder tridiagonalization of $A$,
  rescaling rows and columns by $-1$ as necessary, yields $U A V' = T$
  with orthogonal matrices $U$, $V$ and
\[
T =  \frac{1}{\sqrt{2}}
  \begin{pmatrix}
    0 & \tau_{n-1} \\
    \tau_{n-1} & 0 & \tau_{n-2} \\
    &\tau_{n-2} & 0 & \tau_{n-3} \\
    &&\ddots & \ddots & \ddots \\
    &&&\tau_2 & 0 & \tau_1\\
    &&&& \tau_1 & 0
  \end{pmatrix},
\]
where the entries $\tau_k$ are jointly distributed as independent
$\chi_k$-variables with degrees of freedom $k=1,\ldots,n$, see
\cite[Sect. II]{Dumitriu}.  Since Householder triadiagonalizations do
not operate on the first row and column, we have $Ue_1 = e_1$ and,
hence,
\[
U \cdot (\tau_n e_1 \;\; A) \cdot 
\begin{pmatrix}
1 & 0 \\
0 & V'
\end{pmatrix} = (\tau_n e_1 \;\; T).
\]
Thus the matrices $(\tau_n e_1 \;\; A)$ and $(\tau_n e_1 \;\; T)$ have
the same singular values. By a simultaneous row and column permutation
of $T$ so that the odd columns and rows occur before the even ones, we
see that the matrices $(\tau_n e_1\;\; T)$ and (with the
length of the first unit vector $e_1$ adjusted)
\[
\begin{pmatrix}
\tau_n e_1 & 0 & B_\mu^{\even}  \\*[2mm]
0 & (B_\mu^{\even}) ' & 0 
\end{pmatrix}
\]
have the same singular values. Interlacing  shows that the singular values of $B_\mu^{\even} $ correspond to the even ones
of $(\tau_n e_1\;\; A)$ and the singular values of $(\tau_n e_1\;\; B_\mu^{\even})$ correspond to the odd ones. 
\end{proof}

  \begin{remark} 
    The mapping from the singular values of a sample drawn from the GOE to $\tau$ coordinates is deterministic and can be made
    explicit (the same remark applies to the construction of the $\xi$ variables in the next section): starting with $(r,s)$ coordinates, the $\tau$ are
    obtained by applying appropriate orthogonal row and column transformations to the
    matrix 
    \[
    \begin{pmatrix}
    0 & 0 & S'\\
    r &-S&0\\
    \end{pmatrix},\quad\text{where}\quad
    S = \begin{pmatrix}
    s_1 & & \\
    & \ddots &\\
    && s_m\\
    0 &\cdots & 0
    \end{pmatrix};
    \]
 if $\mu=0$, the last row of $S$ is missing.
  \end{remark}    


%% file: DeterminantApplications.tex

\section{Square Bidiagonal Matrix Models and the Determinant}\label{section:determinant}

To study the distribution of determinants we turn the bidiagonal
random matrix model of Corollary~\ref{cor:bidiagonal} into one with
square matrices only. Key to this transformation is the following
variant of a result by Dumitriu and Forrester \cite[Claim
6.5]{Dumitriu}.

\begin{lemma}\label{lem:square}
  Let the variables $\tau_k$ ($k=1,\ldots,2m-1$) be distributed as
  $\chi_k$, with the distribution of $\tau_{2m}$ arbitrary such that
  $\tau_1,\dotsc,\tau_{2m}$ are independent of each other. Then the
  singular values of the $m\times (m+1)$ bidiagonal matrix
\begin{equation}\label{eqn:Bmatrix}
B = 
  \begin{pmatrix}
    \tau_{2m} & \tau_{2m-1}  \\
    & \tau_{2m-2} & \tau_{2m-3} \\
    & &\ddots & \ddots  \\
    & & &\tau_2  & \tau_1
  \end{pmatrix}
\end{equation}
are the same as those of the $m\times m$ bidiagonal matrix
\begin{equation}\label{eqn:Rfactor}
R = \begin{pmatrix}
    \xi_{2m+1} & \xi_{2m-2}  \\
    & \xi_{2m-1} & \xi_{2m-4} \\
    & &\ddots & \ddots  \\
    & & &\xi_5  & \xi_2 \\
    & & & & \xi_3
  \end{pmatrix}
  \end{equation}
constructed by the normalized reduced $RQ$-decomposition\footnote{The $R$-factor can equivalently be obtained from the Cholesky-type decomposition $RR'=BB'$.} $B = R Q$ with a row-orthogonal matrix~$Q$, that is, by the almost
surely positive solution of the set of equations
\begin{subequations}\label{eqn:RQ}
\begin{align}
\xi_{2k+1}^2 + \xi_{2k-2}^2 &= \tau_{2k}^2+\tau_{2k-1}^2 \qquad (k=1,\ldots,m),\label{eqn:RQ_a}\\*[2mm]
\xi_{2k+1}\xi_{2k} &= \tau_{2k+1}\tau_{2k}\qquad (k=1,\ldots,m-1).
\end{align}
\end{subequations}
The variables $\xi_2,\ldots,\xi_{2m-1}$ are distributed as $\chi_2,\ldots,\chi_{2m-1}$; they are independent of each other and of $\tau_{2m}$.
The variable $\xi_{2m+1}$ is of the form
\[
\xi_{2m+1} = \sqrt{\xi_1^2+\tau_{2m}^2}, \quad\text{where}\quad
\xi_1 = \sqrt{\tau_{2m-1}^2-\xi_{2m-2}^2}
\]
is distributed as $\chi_1$ and is also independent of $\xi_2,\ldots,\xi_{2m-1}$ and of $\tau_{2m}$.
\end{lemma}
\begin{proof} A well-known result \cite[Thm.~IX.3.1]{Devroye} about the $\chi^2$-distribution states that the
involution
\begin{equation}\label{eqn:involution}
\phi(X,Y,Z) = \left(Z \frac{X}{X+Y},Z \frac{Y}{X+Y},X+Y\right)
\end{equation}
maps a set of mutually independent random variables $X$, $Y$, $Z$ distributed as $\chi_r^2$, $\chi_s^2$ and $\chi_{r+s}^2$ to a new
set of mutually independent random variables of exactly the same type. Starting with $\tau_{1,1}=\tau_1$, the system \eqref{eqn:RQ} is recursively solved for the variables $\xi_2,\ldots,\xi_{2m-1}$ by
\[
(\tau_{1,k+1}^2,\xi_{2k}^2,\xi_{2k+1}^2) = \phi(\tau_{1,k}^2,\tau_{2k}^2,\tau_{2k+1}^2),\qquad (k=1,\ldots,m-1).
\]
Hence, the variable $\xi_1 = \tau_{1,m}$ and the thus constructed $\xi_2,\ldots,\xi_{2m-1}$ are independent of each other and
of the not yet used variable $\tau_{2m}$; they are distributed as $\chi_k$ ($k=1,\ldots,2m-1$). Because $\phi$ is an
involution, there is
\[
\tau_{2m-1}^2 = \tau_{1,m}^2 + \xi_{2m-2}^2 = \xi_1^2 + \xi_{2m-2}^2.
\]
Hence, the yet to be used $k=m$ case of Eq.~\eqref{eqn:RQ_a} finally implies the asserted form of $\xi_{2m+1}$.
\end{proof}

\begin{remark} The use of the involution \eqref{eqn:involution} has been motivated by the observation \cite[Lemma~1]{EL-GUEsing} that for $2\times 2$ matrices the
$R$-factor of the $RQ$-decomposition
of a lower triangular matrix
\[
L = \begin{pmatrix}
z & 0\\
y & x
\end{pmatrix} = RQ,\qquad R = 
\begin{pmatrix}
\xi & \eta\\
0 & \zeta
\end{pmatrix},
\]
is induced by the transformation $(\xi^2,\eta^2,\zeta^2) = \phi(x^2,y^2,z^2)$. Basically, the $R$-factor \eqref{eqn:Rfactor} of the bidiagonal matrix \eqref{eqn:Bmatrix} is then
obtained by successively applying this transformation along the diagonal from the lower right to the upper left.
\end{remark}

Application of this lemma to Corollary~\ref{cor:bidiagonal} yields sparse random matrix models for
the odd and even  singular values of the \GOE{} in terms of bidiagonal \emph{square} matrices. We begin with
the case of even order $n=2m$. 

\begin{theorem}\label{thm:square-bidiagonal-even} Let $\xi_1,\ldots,\xi_{2m}$ be independent random variables, with $\xi_k$ distributed as~$\chi_k$. The union of the singular values of the two bidiagonal square matrices
\begin{gather*}
R_0^{\odd}  = 
\frac{1}{\sqrt{2}}\begin{pmatrix}
    \sqrt{\xi_1^2 + 2\xi_{2m}^2} & \xi_{2m-2}  \\
    & \xi_{2m-1} & \xi_{2m-4} \\
    & &\ddots & \ddots  \\
    & & &\xi_5  & \xi_2 \\
    & & & & \xi_3
  \end{pmatrix}\\
\intertext{and}
R_0^{\even} = 
\frac{1}{\sqrt{2}}
  \begin{pmatrix}
    \xi_1 & \xi_{2m-2}  \\
    & \xi_{2m-1} & \xi_{2m-4} \\
    & &\ddots & \ddots  \\
    & & &\xi_5  & \xi_2 \\
    & & & & \xi_3
  \end{pmatrix}
  \end{gather*}
is jointly distributed as $|\GOE_{2m}|$. Here, the singular values of $R_0^{\odd}$ correspond to
$\odd|\GOE_{2m}|$ and the singular values of $R_0^{\even}$ correspond to $\even|\GOE_{2m}|$, both drawn from the same
ensemble. 
\end{theorem}
\begin{proof} Let $\tau_1,\ldots,\tau_{2m}$ be independent random variables, with $\tau_k$ distributed as $\chi_k$.
Prepending a zero column to the second matrix,
Corollary~\ref{cor:bidiagonal} shows that the singular values of the two matrices
\begin{gather*}
B_0^{\odd} = \frac{1}{\sqrt{2}}
  \begin{pmatrix}
    \tau_{2m}^{\odd} & \tau_{2m-1}  \\
    & \tau_{2m-2} & \tau_{2m-3} \\
    & &\ddots & \ddots  \\
    & & &\tau_2  & \tau_1
  \end{pmatrix}\\
  \intertext{and}
B_0^{\even} = \frac{1}{\sqrt{2}}
  \begin{pmatrix}
   \tau_{2m}^{\even} & \tau_{2m-1}  \\
     & \tau_{2m-2} & \tau_{2m-3} \\
     & &\ddots & \ddots  \\
     & & &\tau_2  & \tau_1
  \end{pmatrix},
\end{gather*}
where
\[
\tau_{2m}^{\odd} = \sqrt{2}\tau_{2m},\qquad \tau_{2m}^{\even} =0,
\]
are jointly distributed as $|\GOE_{2m}|$. Here, the singular values of $B_0^{\odd}$ correspond to
$\odd|\GOE_{2m}|$ and the singular values of $B_0^{\even}$ correspond to $\even|\GOE_{2m}|$, both drawn from the same
ensemble. If we apply the construction of Lemma~\ref{lem:square} to both matrices $B_0^{\odd}$ and $B_0^{\even}$ 
simultaneously, we obtain the $R$-factors $R_0^{\odd}$ and $R_0^{\even}$ with
one and the same set of variables $\xi_1,\ldots,\xi_{2m}$ subject to the asserted properties and additionally
\[
\xi_{2m+1}^{\odd} = \sqrt{\xi_1^2 + (\tau_{2m}^{\odd})^2} = \sqrt{\xi_1^2 + 2\tau_{2m}^2} ,\qquad \xi_{2m+1}^{\even} 
= \sqrt{\xi_1^2 + (\tau_{2m}^{\even})^2}= \xi_1.
\]
By defining $\xi_{2m}=\tau_{2m}$ we thus get the asserted form of $R_0^{\odd}$ and $R_0^{\even}$. 
\end{proof}

\begin{remark}
  The matrices $R_0^{\even}$ of
  Theorem~\ref{thm:square-bidiagonal-even} and $B_0^{\even}$ of
  Corollary~\ref{cor:bidiagonal} are superficially related, up to a
  different sample of the independent random variables, by a transposition
  followed by a cyclic permutation of their diagonals.  Such a
  transformation would not, in general, preserve singular values, but
  depends instead on properties of the $\chi$ distributions.
\end{remark}

The case of odd order exhibits a similar structure. The equivalence between the models $B_1^{\even}$ and $R_1^{\even}$
  used in the following theorem is also noted, for the anti-GUE, in
  \cite[Sect.~2]{EL-GUEsing} and \cite[Clm.~6.5]{Dumitriu}.

\begin{theorem}\label{thm:square-bidiagonal-odd} Let $\xi_1,\ldots,\xi_{2m+1}$ be independent random variables, with $\xi_k$ distributed as~$\chi_k$. The union of the singular values of the two bidiagonal square matrices
\[
R_1^{\odd}  = 
\frac{1}{\sqrt{2}}\begin{pmatrix}
    \sqrt{2} \xi_1  & \sqrt{2}\xi_{2m}  \\
    & \xi_{2m+1} & \xi_{2m-2} \\
    & &\ddots & \ddots  \\
    & & &\xi_5  & \xi_2 \\
    & & & & \xi_3
  \end{pmatrix}
\]
and
\[
R_1^{\even} = 
\frac{1}{\sqrt{2}}
  \begin{pmatrix}
    \xi_{2m+1} & \xi_{2m-2}  \\
    & \xi_{2m-1} & \xi_{2m-4} \\
    & &\ddots & \ddots  \\
    & & &\xi_5  & \xi_2 \\
    & & & & \xi_3
  \end{pmatrix}
\]
is jointly distributed as $|\GOE_{2m+1}|$. Here, the singular values of $R_1^{\odd}$ correspond to
$\odd|\GOE_{2m+1}|$ and the singular values of $R_1^{\even}$ correspond to $\even|\GOE_{2m+1}|$, both drawn from the same
ensemble. 
\end{theorem}

\begin{proof} Let $\tau_1,\ldots,\tau_{2m+1}$ be independent random variables, with $\tau_k$ distributed as $\chi_k$.
By transposing both matrices and prepending a zero column to the first one,
Corollary~\ref{cor:bidiagonal} shows that the singular values of the two matrices
\begin{gather*}
B_1^{\odd} = \frac{1}{\sqrt{2}}
    \begin{pmatrix}
    0 & \sqrt{2}\tau_{2m+1}\\
    &\tau_{2m} & \tau_{2m-1}  \\
    & &\ddots & \ddots  \\
    & & &\tau_2  & \tau_1
  \end{pmatrix}\\
\intertext{and}
B_1^{\even} = \frac{1}{\sqrt{2}}
    \begin{pmatrix}
    \tau_{2m} & \tau_{2m-1}  \\
    & \tau_{2m-2} & \tau_{2m-3} \\
    & &\ddots & \ddots  \\
    & & &\tau_2  & \tau_1
  \end{pmatrix}
\end{gather*}
are jointly distributed as $|\GOE_{2m+1}|$. Here, the singular values of $B_0^{\odd}$ correspond to
$\odd|\GOE_{2m+1}|$ and the singular values of $B_0^{\even}$ correspond to $\even|\GOE_{2m+1}|$, both drawn from the same
ensemble. If we apply the construction of Lemma~\ref{lem:square} simultaneously to $B_1^{\even}$ and to
\[
 \hat B_1^{\odd} = \frac{1}{\sqrt{2}}
    \begin{pmatrix}
    \tau_{2m+2}^{\odd} & \tau_{2m+1}\\
    &\tau_{2m} & \tau_{2m-1}  \\
    & &\ddots & \ddots  \\
    & & &\tau_2  & \tau_1
  \end{pmatrix},\qquad \tau_{2m+2}^{\odd} = 0,
 \]
 which is $B_1^{\odd}$ with its first row rescaled, we obtain the $R$-factors $R_1^{\even}$ and
\[
\hat R_1^{\odd} = 
\frac{1}{\sqrt{2}}\begin{pmatrix}
     \xi_{2m+3}^{\odd}  & \xi_{2m}  \\
    & \xi_{2m+1} & \xi_{2m-2} \\
    & &\ddots & \ddots  \\
    & & &\xi_5  & \xi_2 \\
    & & & & \xi_3
  \end{pmatrix}
\]
with one and the same set of variables $\xi_{1},\ldots,\xi_{2m+1}$
subject to the asserted properties and additionally
\[
\xi_{2m+3}^{\odd} = \sqrt{\xi_1^2+(\tau_{2m+2}^{\odd})^2} = \xi_1.
\]
By restoring the proper scaling of the first row, we thus get the
asserted form of $R_1^{\even}$ and $R_1^{\odd}$.
\end{proof}

\begin{remark} This theorem implies that, for odd order $n=2m+1$, the product
  $\det R_1^{\odd}$ of the odd singular values of $\GOE_n$ and the product
  $\det R_1^{\even}$ of the even ones are related by
  \[
  \det R_1^{\odd} = \xi_1 \det R_1^{\even},
  \]
  where $\xi_1$ is a random variable distributed as $\chi_1$ which is
  independent of $\even |\GOE_n|$. This is nothing but (\ref{eqn:t_prod}),
  recalling that by the proof of
  Theorem~\ref{thm:independent} the variable $r_{m+1}$ is
  distributed as $\chi_1$.
\end{remark}

As an immediate consequence of the preceding theorems, stated in the following corollary, the magnitude of the determinant of the GOE can be
expressed as a product of independent random variables. Even though Tao and Vu
  speculated that such a representation does not seem to be possible \cite[p.~78]{TaoVu},
a precursor of this
  result was recognized implicitly by Delannay and Le~Ca\"er
  who noted in \cite[p.~1531]{Delannay-LeCaer} that the Meijer $G$-function 
  they used to describe the distribution of the determinant
  when $n$ is odd could be sampled as a product of independent
  gamma-distributed random variables.  They did not, however, have any
  interpretation for these variables in terms of the underlying
  ensemble, nor did they recognize the possibility of sampling when
  $n$ is even.

\begin{corollary}\label{cor:detfact} Let $G_n$ be drawn from the $\GOE$ of
  order $n=2m+\mu$ with parity $\mu=0,1$, $\hat{m}=m+\mu$. Then the
  determinant of $M_n = \sqrt{2} G_n$ factors into independent random
  variables of the form
\begin{equation}\label{eqn:abs-det}
|\!\det M_{n}| = \eta_n^{(1)} \cdot \xi_3^2 \cdot \xi_5^2 \,\cdots\, \xi_{2\hat{m}-1}^2,
\end{equation}
with 
\[
 \eta_n^{(1)} = \xi_1\cdot \sqrt{\xi_1^2+2\xi_{n}^2} \quad (\mu=0),\qquad\quad \eta_n^{(1)}= \sqrt{2}\, \xi_1 \quad (\mu=1).
\]
Here, the $\xi_k$ are mutually independent random variables distributed as
$\chi_k$.
\end{corollary}
\begin{proof} The assertion follows from the observation that
  $|\!\det{}M_n|$ is the product of the singular values of $M_n$ and
  therefore, by Theorem~\ref{thm:square-bidiagonal-even}
  and~\ref{thm:square-bidiagonal-odd}, distributed as
\[
\det(\sqrt{2} R_\mu^{\even}) \cdot \det(\sqrt{2} R_\mu^{\odd}).
\]
Multiplication of the diagonal terms of the bidiagonal factors finishes the proof.
\end{proof}

In retrospect, once we know that  $|\!\det\GOE_n|$ is distributed as
a product of $\hat{m}$ independent random variables, all the
factors can be readily identified in the Mellin transforms computed
by Delannay and Le~Ca\"er in \cite[Eqs.~(26/41))]{Delannay-LeCaer}, although, for even $n=2m$, the Mellin
transform of the factor
\[
\eta_{2m}^{(1)} = \xi_1\sqrt{\xi_1^2+2\xi_{2m}^2}
\]
into the expression
\begin{equation}\label{eqn:fancy-mellin}
2^{3(s-1)/2} \,
\frac{\Gamma\left(\tfrac{s}{2}\right)\Gamma\left(s+m-\tfrac{1}{2} \right)}{\Gamma\left(\tfrac12 \right)\Gamma\left(\frac{s}{2}+m \right)} 
\; {}_2F_1\!\left(\frac{s}{2},\frac{1-s}{2};\frac{s}{2}+m;\frac{1}{2}\right)
\end{equation}
with the hypergeometric function ${}_2 F_1$ may not be familiar to most observers.  This aspect of their paper has been
  missed by several commentators, with Mehta, in \cite[\S26.6]{Mehta},
  omitting their expression for even $n=2m$ since the inverse Mellin
  transform of \eqref{eqn:fancy-mellin} cannot be readily written down.   
\begin{remark}
  In the case of odd order $n=2m+1$, the density of $\det\GOE_{2m+1}$ is
  necessarily odd, since the eigenvalue density \eqref{eqn:goedensity} is even, but
  $\det(G)=-\det(-G)$.  It follows that in this case the sign of the
  determinant is statistically independent of its magnitude, and we can obtain the
  distribution of the determinant by replacing $\xi_1$ by a
  standard normal variable.  No corresponding result is available for
  even order $n=2m$,
    although the factored presentation of the odd moments of
    $\det\GOE_{2m}$ in \cite[Eq.~(23)]{Andrews-Goulden-Jackson}
    suggests that the distribution of the determinant should involve
    many of the same factors.
\end{remark}

%% file: DeterminantCLT.tex

\section{Central Limit Theorem for the Determinant}\label{section:CLT}

Delannay and Le~Ca\"er used an explicit computation of the Mellin
transform of the even part of the distribution of $\det{\GOE_n}$ to
derive the cumulants of the potential $V=\log\abs{\det{\GOE_n}}$, and
to show that $V$ is asymptotically Gaussian
\cite[Section~III]{Delannay-LeCaer}.  Tao and Vu extended 
this log-normality to determinants of a wider class of
Wigner matrices, and provided an alternate proof in the case of
Gaussian matrices in \cite{TaoVu}:  based on analyzing tridiagonal sparse models for the \GOE{}
and \GUE{} eigenvalues, they found a way to approximate the log-determinant as a sum of {\em weakly dependent} terms,
which then yields the asymptotic log-normality by stochastic calculus and the martingale central limit theorem.
In this section we present yet another, much simpler proof based on
the factorization of the magnitude of the determinant into independent
random variables. In particular, our proof elucidates the difference
between the GOE and the GUE in the scaling of the central limit
theorem.

We start by recalling that, parallel to the factorization given in
Corollary~\ref{cor:detfact} for the GOE ($\beta=1$), Edelman and
La~Croix \cite[Thm.~2]{EL-GUEsing} obtained a factorization for the
GUE ($\beta=2$): with $G_n$ drawn from the $\GUE$ of order $n$, the
determinant of $M_n = \sqrt{2} G_n$ factors into independent random
variables of the form
\begin{equation}\label{eqn:GUE-det}
|\!\det M_{n}| = \eta_n^{(2)} \cdot \xi_3 \tilde\xi_3 \cdot \xi_5 \tilde\xi_5 \,\cdots \, \xi_{2\hat{m}-1} \tilde\xi_{2\hat{m}-1} 
\end{equation}
with
\[
\eta_n^{(2)} = \xi_1 \xi_{n+1} \quad (\mu=0),\qquad  \eta_n^{(2)} = \xi_1 \quad (\mu=1).
\]
Here $\xi_1,\ldots,\xi_n,\tilde\xi_3,\ldots,\tilde\xi_{2\hat{m}-1}$ are mutually independent random variables with both $\xi_{k}$
and $\tilde\xi_k$ being distributed as $\chi_k$. Note that, except for the (asymptotically irrelevant) change in the 
factor $\eta_n^{(\beta)}$, the transition from the GOE to the GUE
just amounts for splitting the terms $\xi_k^2$ into the products $\xi_k\tilde\xi_k$ of independent factors. It is precisely
this split which causes the appearance of $\beta$ in the denominator of the central limit theorem when written in the
following form.

\begin{theorem}[Tao and Vu \protect{\cite[Thm.~4]{TaoVu}}] With the notation as above there holds, as $n\to \infty$, the central limit theorem
\begin{equation}\label{eqn:det-CLT}
\frac{\log|\!\det M_n| - \frac{1}{2}\log n!  + \frac{1}{4} \log n}{\sqrt{\frac{1}{\beta}\log n}} \;\overset{\rm d}{\to}\; N(0,1) \qquad(\beta=1,2),
\end{equation}
where $\overset{\rm d}{\to}$ denotes convergence in distribution.
\end{theorem}

To prove this theorem for $\beta=1$ and $\beta=2$ in parallel, we split
\[
\log |\!\det M_n | = Y_n^{(\beta)} + Z_n^{(\beta)}\qquad (\beta=1,2)
\]
into the random variables $Y_n^{(\beta)} = \log \eta_n^{(\beta)}$ and $Z_n^{(\beta)}$ defined by
\begin{align*}
Z_n^{(1)} &=  \quad\;\;\;\; 2\log \xi_3  \quad\;\;\;\, + \quad\;\;\;\; 2\log \xi_5 \quad\;\;\;\,  +\quad \cdots\quad + \quad\;\;\;\,  2\log \xi_{2\hat{m}-1},\\*[2mm]
Z_n^{(2)} &=  \left(\log \xi_3 + \log \tilde \xi_3\right) + \left(\log \xi_5+ \log \tilde \xi_5\right) + \quad\cdots\quad + \left(\log \xi_{2\hat{m}-1} + \log \tilde\xi_{2\hat{m}-1}\right),
\end{align*}
We immediately observe the relations
\begin{equation}\label{eqn:meanvar}
\E{Z_n^{(1)}} = \E{Z_n^{(2)}}, \qquad \Var{Z_n^{(1)}} = 2\Var{Z_n^{(2)}}.
\end{equation}
Note that the factor of two between the variances is caused,  in the transition from GOE to GUE, by the above mentioned split of  $\xi_k^2$ into the product $\xi_k\tilde\xi_k$.

Now, while proving the central limit theorem in the $\beta=2$ case, Edelman and La Croix \cite[Cor.~2]{EL-GUEsing} obtained, in passing,
the following result. 

\begin{lemma} The random variable $Z_n^{(\beta)}$ satisfies, as $n\to\infty$, a central limit theorem of the form
\begin{equation}\label{eqn:Z_CLT}
\tilde Z_n^{(\beta)} = \frac{Z_n^{(\beta)} - \frac{1}{2}\log n! + \frac{1-\mu}{2}\log n + \frac{1}{4} \log n}{\sqrt{\frac{1}{\beta}\log n}}
\;\overset{\rm d}{\to}\; N(0,1)
\qquad (\beta=1,2).
\end{equation}
\end{lemma}

\begin{proof} The proof of \cite[Cor.~2]{EL-GUEsing} proceeds, first, by establishing  asymptotic expansions
based on explicit calculations 
of the mean and variance of $\log\chi$-distributed variables, namely,
\begin{align*}
\E{Z_n^{(2)}} &= \frac{1}{2}\log (2\hat{m}-1)! \,-\, \frac{1-\mu}{2}\log (2\hat{m}-1) \,-\, \frac{1}{4} \log (2\hat{m}-1)  \,+\, O(1),\\*[2mm] 
\Var{Z_n^{(2)}} &= \frac{1}{2}\log (2\hat{m}-1)  \,+\, O(1),
\end{align*}
and, next, by showing that $Z_n^{(2)}$ satisfies a Lyapunov condition of order four. Hence, the Lindeberg--Feller central limit theorem can then be applied to $Z_n^{(2)}$ and gives, by noting that
\[
\log(2\hat{m} - 1)! = \log n! \,-\, (1- \mu)\log n, \qquad \log (2\hat{m}-1) = \log n \,+\, O(1),
\]
the asserted limit \eqref{eqn:Z_CLT}. By realizing that the sums $Z_n^{(1)}$ and $Z_n^{(2)}$ basically share the same Lyapunov condition,  the central limit
theorem for $Z_n^{(1)}$ can be induced from that of  $Z_n^{(2)}$ by means of \eqref{eqn:meanvar}.
\end{proof}

The difference between the central limit theorems of $\log |\!\det M_n|$ and of $Z_n^{(\beta)}$ enjoys the following strong convergence result.

\begin{lemma} The random variable $Y_n^{(\beta)}$ satisfies, as $n \to\infty$,
\begin{equation}\label{eqn:Y-as}
\tilde Y_n^{(\beta)} = \frac{Y_n^{(\beta)} - \frac{1-\mu}{2}\log n}{\sqrt{\frac{1}{\beta}\log n}} \;\overset{\rm a.s.}{\longrightarrow}\; 0 \qquad (\mu=0,1),
\end{equation}
where $\overset{\rm a.s.}{\longrightarrow}$ denotes almost sure convergence.
\end{lemma}
\begin{proof} The case $\mu=1$ is trivial, since in that case $Y_n^{(\beta)}$ is independent of $n$.
Applied to a sum of squares of independent standard Gaussians, the strong law of large numbers gives, as $n\to\infty$,
\[
n^{-1} \xi_n^2 \overset{\rm a.s.}{\longrightarrow} \E{\xi_1^2} = 1,\quad\text{and, hence,}\quad
n^{-1} (\xi_1^2 + 2 \xi_n^2)  \overset{\rm a.s.}{\longrightarrow} 2.
\]
Taking the logarithm gives
\[
\log \xi_n - \frac{1}{2}\log n  \overset{\rm a.s.}{\longrightarrow} 0,\qquad 
\log \sqrt{\xi_1^2+2\xi_n^2} - \frac{1}{2}\log n  \overset{\rm a.s.}{\longrightarrow} \frac{1}{2}\log 2,
\]
which implies the assertion for $\mu=0$. 
\end{proof}

Now, adding \eqref{eqn:Z_CLT} and \eqref{eqn:Y-as} gives, by Slutsky's theorem, 
\[
\frac{\log|\!\det M_n| - \frac{1}{2}\log n!  + \frac{1}{4} \log n}{\sqrt{\frac{1}{\beta}\log n}} = \tilde Y_n^{(\beta)} + \tilde Z_n^{(\beta)} 
\;\overset{\rm d}{\to}\; N(0,1) \qquad(\beta=1,2),
\]
which finishes the proof of the central limit theorem \eqref{eqn:det-CLT}.


%% file: OrderedConeAnalysis.tex

\section{Integrating Out the Odd or Even Singular Values}\label{section:evensingularvalues}

Here, we present another proof of Theorem~\ref{thm:main1}. If we are
interested only in the distribution of the even singular values, then
it is possible to proceed from the joint probability
density~(\ref{eqn:singular_density}) by integrating out the odd
ones. While not exposing any additional structure, such as in
Theorem~\ref{thm:independent}, this approach is conceptually more
straight forward, and offers the advantage that it can also be used to
establish the determinantal formula (\ref{eqn:odd_density}) for the
probability density of the odd singular values. This is of interest in
its own right, since we constructed separate sparse random matrix
models for the odd singular values in Corollary~\ref{cor:bidiagonal}
and in Theorems~\ref{thm:square-bidiagonal-even} and
\ref{thm:square-bidiagonal-odd}. Moreover, the technique extends to
the symmetric Jacobi and to the Cauchy ensembles \cite{1503.07383}.

\subsection{Integrating out the odd singular values}
  
Recalling (\ref{eqn:dim}), we rewrite the expression (\ref{eqn:singular_density}) of the joint
density in the form
 \[
 q(s;t) = c_n 2^n n! \cdot 
 g_\mu(s_1,\ldots,s_m) \cdot g_{1-\mu}(t_1,\ldots,t_{\hat{m}})
\]
with functions
\begin{equation}\label{eqn:gmu}
  g_a(z_1,\ldots,z_m) = \prod_{k=1}^m z_k^a e^{-z_k^2/2} \cdot \Delta(z_m^2,\ldots,z_1^2).
\end{equation}
Corollary~\ref{cor:2} below shows that integrating out the odd singular
values $t$ subject to the interlacing~\eqref{eqn:interlacing} gives
the following marginal density of the even singular values with
$\delta_\mu = (\pi/2)^{\mu/2}$:
\begin{multline}\label{eqn:even_density}
q_{\text{even}}(s_1,\ldots,s_m) = \delta_\mu c_n 2^n n! \cdot g_\mu(s_1,\ldots,s_m)^2\\*[2mm]
= \delta_\mu c_n 2^n n! \cdot\prod_{k=1}^m s_k^{2\mu} e^{-s_k^2} \cdot \Delta(s_m^2,\ldots,s_1^2)^2.
\end{multline}
Since the last expression is the joint density of the
anti-\GUE{} of order $n$, see \cite[Sect.~13.1]{Mehta} or \cite[Ex. 1.3.5(iv)]{ForresterBook}, this is nothing but Theorem~\ref{thm:main1} spelled out in terms of densities.

The integration is based on the following lemma and its first Corollary~\ref{cor:2}.

\begin{lemma}\label{lem:1} Let
\[
e_\kappa^{(n)}(x) = 
\begin{pmatrix}
x^{\kappa} e^{-x^2/2}\\
x^{\kappa +2}e^{-x^2/2}\\
\vdots\\
x^{\kappa +2n-2}e^{-x^2/2}
\end{pmatrix}
  \in \R^n \qquad (\kappa = -1,0,1),
\]
with the understanding that, instead of $x^{-1} e^{-x^2/2}$, the first entry of $e_{-1}^{(n)}(x)$ is the expression
\[
\eta_{-1}(x) = - \sqrt{\frac{\pi}2} \erf\left(\frac{x}{\sqrt{2}}\right).
\]
Then, for $\kappa=0,1$, there holds the integration formula
\begin{multline*}
\int_{x_1}^{x_2}d \xi_1 \cdots \int_{x_n}^{x_{n+1}}d \xi_n \;\det\left(e_\kappa ^{(n)}(\xi_1) \; \cdots \; e_\kappa^{(n)}(\xi_n)\right)\\*[2mm]
 = 
\det
\begin{pmatrix}
e_{\kappa-1}^{(n)}(x_1) & \cdots & e_{\kappa-1}^{(n)}(x_{n+1}) \\*[1mm]
1 & \cdots & 1
\end{pmatrix}.
\end{multline*}
\end{lemma}

\begin{proof}
Integration by parts yields the three-term recurrence of antiderivatives
\begin{align*}
\int^x  e^{-\xi^2/2}\,d\xi &= - \eta_{-1}(x),\\*[2mm]
 \int^x \xi^{k+1} e^{-\xi^2/2}\,d\xi &= -x^k e^{-x^2/2} +k \int^x \xi^{k-1}   e^{-\xi^2/2}\,d\xi \qquad (k=0,1,2,\ldots),
\end{align*}
and, hence, by simplifying notation to $e_\kappa(x)=e_\kappa^{(n)}(x)$,
\[
\int^x e_\kappa(\xi) \,d\xi = L_\kappa e_{\kappa-1} (x) \qquad (\kappa=0,1)
\]
with a \emph{lower triangular} matrix $L_\kappa \in {\R^{n\times n}}$ having $-1$ all along its main diagonal.
We thus
calculate
\begin{multline*}
\int_{x_1}^{x_2}d \xi_1 \cdots \int_{x_n}^{x_{n+1}}d \xi_n \;\det\left(e_\kappa(\xi_1) \; \cdots \; e_\kappa(\xi_n)\right)\\*[4mm]
 = \det\left(\int_{x_1}^{x_2} e_\kappa(\xi_1) d\xi_1 \; \cdots \; \int_{x_n}^{x_{n+1}} e_\kappa(\xi_n) d\xi_{n}\right)\\*[2mm]
=\underbrace{ \det(L_\kappa)}_{=(-1)^n} \det\left(e_{\kappa-1}\Big|_{x_1}^{x_2}\;\; \cdots \;\; e_{\kappa-1}\Big|_{x_n}^{x_{n+1}}\right)\\*[2mm]
= \det
\begin{pmatrix}
e_{\kappa-1}(x_1) & e_{\kappa-1}\Big|_{x_1}^{x_2} & \cdots & e_{\kappa-1}\Big|_{x_n}^{x_{n+1}} \\*[3mm]
1 & 0 & \cdots & 0
\end{pmatrix}\\*[2mm]
= \det
\begin{pmatrix}
e_{\kappa-1}(x_1) & e_{\kappa-1}(x_2) &\cdots & e_{\kappa-1}(x_{n+1}) \\*[2mm]
1 & 1 & \cdots & 1
\end{pmatrix}.
\end{multline*}
In the last step we added the first column to the second, then
 the second to the third, etc.
\end{proof}

\begin{corollary}\label{cor:2} Let $g_\mu$ be as in \eqref{eqn:gmu} and put $s_{\hat m}=0$ if $\mu=1$.
Then, one has the integration formula
\[
\int_{s_1}^{\infty} dt_1 \int_{s_2}^{s_1} dt_2\cdots \int_{s_{\hat m}}^{s_{\hat{m}-1}} dt_{\hat{m}}\, g_{1-\mu}(t_1,\ldots,t_{\hat m}) = \delta_\mu \, g_\mu(s_1,\ldots,s_m)\qquad (\mu=0,1)
\]
with $\delta_\mu = \left(\pi/2\right)^{\mu/2}$.
\end{corollary}

\begin{proof} Using the notation of Lemma~\ref{lem:1}, we first observe that 
\begin{equation}\label{eqn:gmu_as_det}
g_\mu(z_1,\ldots,z_m) = \det\left(e^{(m)}_\mu(z_m)\;\cdots\; e^{(m)}_\mu(z_1)\right).
\end{equation}
Now, Lemma~\ref{lem:1} yields, first using $e_0^{(m)}(\infty)=0$, that for $\mu=0$
\begin{multline*}
\int_{s_1}^{\infty} dt_1 \int_{s_2}^{s_1} dt_2\cdots \int_{s_m}^{s_{m-1}} dt_m\, \,\det\left(e_1^{(m)}(t_m)\;\cdots\; e_1^{(m)}(t_1)\right)\\*[2mm]
 = 
\det
\begin{pmatrix}
e_{0}^{(m)}(s_m) & \cdots & e_{0}^{(m)}(s_{1}) & 0 \\*[1mm]
1 &  \cdots & 1 & 1
\end{pmatrix}
 = 
\det\left(e_{0}^{(m)}(s_m)\;\cdots\;e_{0}^{(m)}(s_{1})\right)
\end{multline*}
and then, using $e_{-1}^{(m+1)}(0)=0$ and $e_{-1}^{(m+1)}(\infty) = -(\sqrt{\pi/2},0)'$, that for $\mu=1$
\begin{multline*}
\int_{s_1}^{\infty} dt_1 \int_{s_2}^{s_1} dt_2\cdots \int_{0}^{s_m} dt_{m+1} \, \det\left(e_0^{(m+1)}(t_{m+1})\;\cdots\; e_0^{(m+1)}(t_1)\right) \\*[2mm]
 = 
\det
\begin{pmatrix}
0 & e_{-1}^{(m+1)}(s_m) & \cdots & e_{-1}^{(m+1)}(s_{1}) & e_{-1}^{(m+1)}(\infty)\\*[1mm]
1 & 1 & \cdots & 1 & 1
\end{pmatrix}\\*[2mm]
=
(-1)^m\det
\begin{pmatrix}
 e_{-1}^{(m+1)}(s_m) & \cdots & e_{-1}^{(m+1)}(s_1) & e_{-1}^{(m+1)}(\infty)
\end{pmatrix}\\*[2mm]
=
(-1)^m\det
\begin{pmatrix}
\eta_{-1}(s_m) & \cdots & \eta_{-1}(s_1) & -\sqrt{\frac{\pi}2} \\*[1mm]
 e_{1}^{(m)}(s_m) & \cdots & e_{1}^{(m)}(s_1) & 0
\end{pmatrix}
 = \sqrt{\frac{\pi}2}
 \det\left(e_{1}^{(m)}(s_m)\;\cdots\; e_{1}^{(m)}(s_1)\right),
\end{multline*}
which finishes the proof of the corollary.
\end{proof}

\subsection{Integrating out the even singular values}

The following second corollary of Lemma~\ref{lem:1} will allow us to integrate out the {\em even} singular
values from the density $q(s;t)$.

\begin{corollary}\label{cor:even_out} Let $g_\mu$ be as in \eqref{eqn:gmu} and put $t_{m+1}=0$ if $\mu=0$.
Then, for $\mu=0,1$, one has the integration formula
\[
\int_{t_2}^{t_1} ds_1 \int_{t_3}^{t_2} ds_2\cdots \int_{t_{m+1}}^{t_m} ds_{m}\, g_{\mu}(s_1,\ldots,s_m) = \det
\begin{pmatrix}
e_{1-\mu}^{(\hat{m}-1)}(t_{\hat{m}}) &\cdots & e_{1-\mu}^{(\hat{m}-1)}(t_1) \\*[2mm]
\delta_{1-\mu}(t_{\hat{m}})  & \cdots & \delta_{1-\mu}(t_1)
\end{pmatrix}
\]
with $\delta_0(t)=1$ and $\delta_1(t) = \sqrt{\pi/2}\erf(t/\sqrt{2})$.
\end{corollary}

\begin{proof} Using (\ref{eqn:gmu_as_det}) and Lemma~\ref{lem:1} we obtain
\begin{multline*}
\int_{t_2}^{t_1} ds_1 \cdots \int_{t_{m+1}}^{t_m} ds_{m}\, g_{\mu}(s_1,\ldots,s_m)\\*[2mm] =
\int_{t_2}^{t_1} ds_1 \cdots \int_{t_{m+1}}^{t_m} ds_{m}\,
\det\left(e^{(m)}_\mu(s_m)\;\cdots\; e^{(m)}_\mu(s_1)\right)\\*[2mm]
= \det
\begin{pmatrix}
e_{\mu-1}^{(m)}(t_{m+1}) & \cdots & e_{\mu-1}^{(m)}(t_1) \\*[1mm]
1 & \cdots & 1
\end{pmatrix},
\end{multline*}
which is already the assertion for $\mu=1$. For $\mu=0$, the assertion follows from further calculating
\begin{multline*}
\det
\begin{pmatrix}
e_{-1}^{(m)}(t_{m+1}) & e_{1}^{(m)}(t_{m}) & \cdots & e_{-1}^{(m)}(t_1) \\*[1mm]
1 & 1 & \cdots & 1
\end{pmatrix}
=
\det
\begin{pmatrix}
0 & e_{1}^{(m)}(t_{m}) & \cdots & e_{-1}^{(m)}(t_1) \\*[1mm]
1 & 1 & \cdots & 1
\end{pmatrix} \\*[2mm]
= (-1)^m \det\left(e^{(m)}_{-1}(t_m)\;\cdots\; e^{(m)}_{-1}(t_1)\right) =
(-1)^m\det
\begin{pmatrix}
-\delta_{1}(t_m) & \cdots & -\delta_{1}(t_1)  \\*[1mm]
 e_{1}^{(m-1)}(t_m) & \cdots & e_{1}^{(m-1)}(t_1)
\end{pmatrix}\\*[2mm]
= \det
\begin{pmatrix}
 e_{1}^{(m-1)}(t_m) & \cdots & e_{1}^{(m-1)}(t_1)\\*[1mm]
\delta_{1}(t_m) & \cdots & \delta_{1}(t_1)  \\*[1mm]
\end{pmatrix}
\end{multline*}
which finishes the proof.
\end{proof}
Now, by means of this corollary, the marginal density of the odd singular values supported on $t_1\geq  t_2 \geq \ldots \geq t_{\hat{m}}\geq  0$
is given as  
\begin{multline}\label{eqn:odd_density}
q_{\text{odd}}(t_1,\ldots,t_{\hat{m}}) =
c_n n!2^n \cdot g_{1-\mu}(t_{\hat{m}},\ldots,t_1) \cdot
\det
\begin{pmatrix}
e_{1-\mu}^{(\hat{m}-1)}(t_{\hat{m}}) &\cdots & e_{1-\mu}^{(\hat{m}-1)}(t_1) \\*[2mm]
\delta_{1-\mu}(t_{\hat{m}})  & \cdots & \delta_{1-\mu}(t_1)
\end{pmatrix}\\*[2mm]
=
c_n n!2^n \cdot 
\det
\begin{pmatrix}
e_{1-\mu}^{(\hat{m}-1)}(t_{\hat{m}}) &\cdots & e_{1-\mu}^{(\hat{m}-1)}(t_1) \\*[2mm]
\gamma_{1-\mu}(t_{\hat{m}})  & \cdots & \gamma_{1-\mu}(t_1)
\end{pmatrix} 
\cdot
\det
\begin{pmatrix}
e_{1-\mu}^{(\hat{m}-1)}(t_{\hat{m}}) &\cdots & e_{1-\mu}^{(\hat{m}-1)}(t_1) \\*[2mm]
\delta_{1-\mu}(t_{\hat{m}})  & \cdots & \delta_{1-\mu}(t_1)
\end{pmatrix}
\end{multline}
with
\[
\gamma_\mu(t) = t^{\mu+2\hat{m}-2} e^{-t^2/2},\qquad
\delta_\mu(t) = 
\begin{cases}
1 &  \quad\text{if $\mu=0$,} \\*[2mm]
 \sqrt{\frac{\pi}2} \erf\left(\frac{t}{\sqrt{2}}\right) & \quad\text{if $\mu=1$.}
\end{cases}
\]
Note that the two determinantal factors differ just in their last rows. It is this difference that prevents the expression from
becoming a perfect square, which is in marked contrast with the marginal density of the even singular values.

%% file: GapProbability.tex

\section{Gap Probabilities}\label{section:gapprobabilities}

Theorem~\ref{thm:main1} has an interesting implication in terms of gap probabilities, that is, in terms of the probabilities
\[
E_\text{RMT}^n(k;J),\qquad E_{\text{RMT}}^{\text{limit}}(k;J),
\]
that the interval $J$ contains exactly $k$ eigenvalues drawn from the random matrix ensemble RMT of finite order $n$,
or in some scaling limit. Here, RMT will be the \GOE{}, the $\aGUE$ or the \LUE{} with parameter $a$.

To begin with, by a simple change of coordinates, see \cite[p.~8]{EL-GUEsing}, there holds
\begin{equation}\label{eqn:LUE}
E_{\aGUE}^{2m+\mu}(k;(0,s)) = \left.E_{\LUE}^{m}\big(k;(0,s^2)\big)\right|_{a=\mu-\tfrac12}\qquad (\mu=0,1).
\end{equation}
By looking at pairs of consecutive values it is easy to see that the event that exactly $k$ values of the decimated ensemble $\even|\GOE_{n}|$, $n=2m+\mu$, are contained in $(0,s)$ is given by the union of the events that
exactly $2k+\mu-1$ or that exactly $2k+\mu$ values of $|\GOE_{n}|$ are in that interval. Since these two events are {\em mutually exclusive}
and since the singular values of GOE contained in $(0,s)$ correspond to the eigenvalues
in $(-s,s)$, we thus get from \eqref{eqn:even_aGUE}
and (\ref{eqn:LUE}) proof of 
\begin{multline}\label{eqn:gap}
E^{2m+\mu}_{\GOE}(2k+\mu-1;(-s,s)) + E_{\GOE}^{2m+\mu}(2k+\mu;(-s,s)) = E_{\aGUE}^{2m+\mu}(k;(0,s)) \\*[2mm]
= \left. E_\text{LUE}^{m}(k;(0,s^2))\right|_{a=\mu-\tfrac12} \qquad (\mu=0,1).
\end{multline} 
For even order ($\mu=0$), a first proof of this formula was given by Forrester \cite[Eq.~(1.14)]{EvenSymm}. 
For odd order ($\mu=1$), Forrester communicated to us further proof of the $k=0$ case, a remarkable tour de force extending the techniques from \cite{EvenSymm}
based on generating functions, Pfaffian calculus, and Fredholm determinants---later he was able to use the
same approach to establish the general $k$ case; for this and the extension to the symmetric Jacobi and to the Cauchy ensembles see \cite{1503.07383}.

  We first identified the $\mu=1$ form of \eqref{eqn:gap}
  via a heuristic duality principle based on three~observations.  
First, the LUE of order $m$ and parameter $a = p - m \in \N$ is modeled by the eigenvalues of
$m\times m$-Wishart matrices
$W = X'X$, where the random $p\times m$-matrices $X$ have independent \emph{complex} standard normal entries. 
Second, the eigenvalues of 
$\tilde{W}= XX'$ are those of $W$ padded with $a = p-m$ zeros; that is, the $(k+a)$-th eigenvalue of $\tilde W$  is distributed as the $k$-th eigenvalue of $W$. Last,  since $\tilde W$ is constructed the same way as $W$, but with dimension $\tilde  m = m + a$ and parameter $\tilde a = -a$, we are thus led, at least {\em formally}, to the duality principle
\begin{equation}\label{eqn:heuristic}
\left. E_{\text{LUE}}^{m+\alpha}(k+\alpha;(0,t))\right|_{a=-\alpha} =  \left. E_{\text{LUE}}^{m}(k;(0,t)) \right|_{a=\alpha}.
\end{equation}
Extrapolated to general $\alpha>-1$, it can be taken as a natural definition of an otherwise undefined expression. Now, formally evaluating the $\mu=0$ form of \eqref{eqn:gap} at half-integer values of $m$ and $k$, and invoking the heuristic duality principle \eqref{eqn:heuristic},
led us to predict the $\mu=1$ form. Since it held up under numerical scrutiny, trying to prove this prediction was a key motivation to our present work.

As already noted by Forrester \cite[Eq.~(1.16)]{EvenSymm},
the bulk scaling of $\GOE$ and the hard-edge scaling of LUE allow us to turn (\ref{eqn:gap}), as $n\to\infty$, into the limit
relation
\begin{multline}\label{eqn:bulk}
E_{\GOE}^{\text{bulk}}(2k-1+\mu;(-s,s)) + E_{\GOE}^{\text{bulk}}(2k+\mu;(-s,s)) \\*[2mm]
= E_{\text{LUE}}^{\text{hard}}(k;(0,\pi^2s^2),\mu-\tfrac12)\qquad (\mu = 0,1);
\end{multline}
a remarkable formula previously established by Mehta \cite[Eqs. (7.5.27/29), (20.1.20/21)]{Mehta} using two different, but much more involved methods. 
In contrast to \eqref{eqn:bulk}, which offers many advantages for the numerical calculation of
gap probabilities of the $\GOE$ in the bulk scaling limit \cite[Sect.~5]{Bornemann}, the finite dimensional version \eqref{eqn:gap} is not yet a closed recursion that would allow us to calculate the gap probabilities of the $\GOE$
on symmetric intervals: a complimentary expression evaluating
\[
E_{\GOE}^{2m+\mu}(2k-\mu;(-s,s)) + E_{\GOE}^{2m+\mu}(2k+1-\mu;(-s,s))\qquad (\mu=0,1)
\]
is still missing.  By the same arguments that justify \eqref{eqn:gap}
such an expression would establish the gap probabilities of the
decimation ensemble $\odd |\GOE_n|$, whose joint distribution is given
by \eqref{eqn:odd_density}.

\begin{figure}[t]
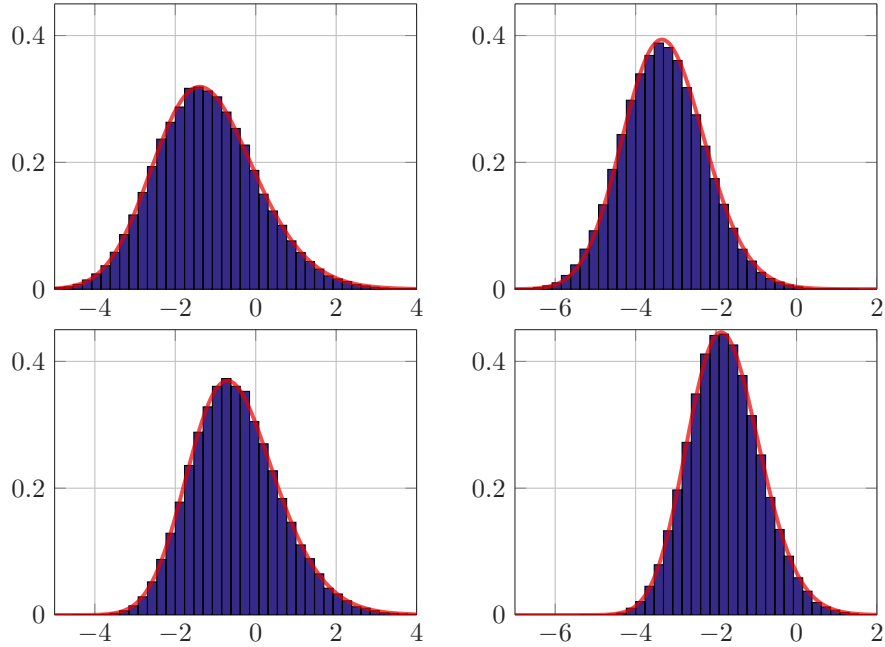

\centering
  \setlength\figurewidth{0.4\textwidth}
  \input{Images/lambda1}\hfil\input{Images/lambda2}\\ 
  \input{Images/sigma1}\hfil\input{Images/sigma2}\\*[-2mm]
  \caption{Fluctuation statistics (densities) of $100\,000$ samples of
    $\GOE_n$, $n=50$, vs. the theoretically predicted soft-edge
    scaling limits; top row: largest eigenvalue vs. $F_1(1;s)$ (left),
    second-largest eigenvalue vs. $F_4(1;s)$ (right); bottom row:
    largest singular value vs. $F_1(1;s)^2$ (left);
    second-largest singular value vs. $F_2(1;s)$ (right). We have
    consistently used the $O(n^{-2/3})$-scaling
    of Johnstone and Ma \protect\cite[Thm.~2]{JohnstoneMa}.
  }\label{fig}
\end{figure}

To finish the paper, it is amusing to note that \emph{all} three cases of the Tracy--Widom distributions 
\[
F_\beta(1;s) = E_\beta^{\text{soft}}(0;(s,\infty)) \qquad (\beta=1,2,4),
\]
can be sampled from the soft-edge scaling limit of the spectrum of just the GOE, i.e., the $\beta=1$ case
(the case $\beta=2$ corresponds to the GUE, $\beta=4$ to the GSE), see Fig.~\ref{fig}.
First, let $\Lambda_1$, $\Lambda_2$ denote
the largest and second-largest soft-edge scaled eigenvalues of the GOE.
In the large-matrix limit, as $n\to\infty$, they are asymptotically
distributed as
\[
\Lambda_1 \overset{\rm d}{\sim} F_1(1;s),\qquad \Lambda_2 \overset{\rm d}{\sim} F_4(1;s).
\]
The first assertion is the definition of the distribution $F_1(1;s)$, while the second follows
from the decimation relation, see \cite[Thm.~5.2]{FR} and \cite[p.~44]{FR2005}, $$\GSE_m = \even \GOE_{2m+1}.$$ Second, let
 $\Sigma_1$, $\Sigma_2$ denote the largest and second-largest 
scaled singular values of the GOE. They are asymptotically distributed as
\[
\Sigma_1 \overset{\rm d}{\sim} F_1(1;s)^2,\qquad \Sigma_2 \overset{\rm d}{\sim} F_2(1;s).
\]
Here, the first assertion follows from the asymptotic independence of the extreme eigenvalues of the GOE and the
second follows from \eqref{eqn:even_aGUE} as follows: $\Sigma_2$ behaves like the largest scaled eigenvalue of the anti-\GUE{}
which, like that of the GUE, is governed by the Tracy--Widom distribution $F_2(1;s)$.

%% file: GOESingularValues.bbl
\providecommand{\bysame}{\leavevmode\hbox to3em{\hrulefill}\thinspace}
\providecommand{\MR}{\relax\ifhmode\unskip\space\fi MR }
\providecommand{\MRhref}[2]{%
  \href{http://www.ams.org/mathscinet-getitem?mr=#1}{#2}
}
\providecommand{\href}[2]{#2}
\begin{thebibliography}{10}

\bibitem{Andrews-Goulden-Jackson}
G.~E. Andrews, I.~P. Goulden, and D.~M. Jackson, \emph{Determinants of random
  matrices and {J}ack polynomials of rectangular shape}, Stud. Appl. Math.
  \textbf{110} (2003), no.~4, 377--390.

\bibitem{Bornemann}
F.~Bornemann, \emph{On the numerical evaluation of distributions in random
  matrix theory: a review}, Markov Process. Related Fields \textbf{16} (2010),
  no.~4, 803--866.

\bibitem{1503.07383}
F.~Bornemann and P.~J. Forrester, \emph{Singular values and evenness symmetry
  in random matrix theory}, 2015, e-print arXiv:1503.07383.

\bibitem{Delannay-LeCaer}
R.~Delannay and G.~Le~Ca{\"e}r, \emph{Distribution of the determinant of a
  random real-symmetric matrix from the {G}aussian orthogonal ensemble}, Phys.
  Rev. E (3) \textbf{62} (2000), no.~2, part A, 1526--1536.

\bibitem{Devroye}
L.~Devroye, \emph{Nonuniform random variate generation}, Springer-Verlag, New
  York, 1986.

\bibitem{Dumitriu}
I.~Dumitriu and P.~J. Forrester, \emph{Tridiagonal realization of the
  antisymmetric {G}aussian {$\beta$}-ensemble}, J. Math. Phys. \textbf{51}
  (2010), no.~9, 093302, 25pp.

\bibitem{EL-GUEsing}
A.~Edelman and M.~La Croix, \emph{The singular values of the {GUE} (less is
  more)}, 2014, e-print arXiv:1410.7065.

\bibitem{EvenSymm}
P.~J. Forrester, \emph{Evenness symmetry and inter-relationships between gap
  probabilities in random matrix theory}, Forum Math. \textbf{18} (2006),
  no.~5, 711--743.

\bibitem{ForresterBook}
\bysame, \emph{Log-gases and random matrices}, Princeton University Press,
  Princeton, NJ, 2010.

\bibitem{FR}
P.~J. Forrester and E.~M. Rains, \emph{Interrelationships between orthogonal,
  unitary and symplectic matrix ensembles}, Random matrix models and their
  applications, Cambridge Univ. Press, Cambridge, 2001, pp.~171--207.

\bibitem{FR2005}
\bysame, \emph{Interpretations of some parameter dependent generalizations of
  classical matrix ensembles}, Probab. Theory Related Fields \textbf{131}
  (2005), no.~1, 1--61.

\bibitem{Golub}
G.~H. Golub and C.~F. Van~Loan, \emph{Matrix computations}, 4th ed., Johns
  Hopkins University Press, Baltimore, 2013.

\bibitem{HornJohnson}
R.~A. Horn and C.~R. Johnson, \emph{Matrix analysis}, 2nd ed., Cambridge
  University Press, Cambridge, 2013.

\bibitem{JohnstoneMa}
I.~M. Johnstone and Z.~Ma, \emph{Fast approach to the {T}racy-{W}idom law at
  the edge of {GOE} and {GUE}}, Ann. Appl. Probab. \textbf{22} (2012), no.~5,
  1962--1988.

\bibitem{Mehta}
M.~L. Mehta, \emph{Random matrices}, 3rd ed., Elsevier/Academic Press,
  Amsterdam, 2004.

\bibitem{Schechter}
S.~Schechter, \emph{On the inversion of certain matrices}, Math. Tables Aids
  Comput. \textbf{13} (1959), 73--77.

\bibitem{TaoVu}
T.~Tao and V.~Vu, \emph{A central limit theorem for the determinant of a
  {W}igner matrix}, Adv. Math. \textbf{231} (2012), no.~1, 74--101.

\bibitem{Thompson}
R.~C. Thompson, \emph{The behavior of eigenvalues and singular values under
  perturbations of restricted rank}, Linear Algebra and Appl. \textbf{13}
  (1976), no.~1/2, 69--78.

\end{thebibliography}
